\documentclass[reqno,11pt,letterpaper]{amsart}

\usepackage[marginratio=1:1, total={6in, 8.5in}]{geometry}

\usepackage[utf8]{inputenc}

\usepackage{amsmath}
\usepackage{amssymb}
\usepackage{amsthm}
\usepackage{mathtools}
\usepackage{graphicx}
\usepackage{comment}
\usepackage[title]{appendix}
\usepackage{hyperref}
\usepackage{csquotes}
\usepackage{enumitem}
\setenumerate[0]{label=(\roman*)}

\newtheorem{thm}{Theorem}[section]
\newtheorem{prop}[thm]{Proposition}
\newtheorem{lemma}[thm]{Lemma}

\theoremstyle{definition}

\theoremstyle{remark}
\newtheorem{rmk}[thm]{Remark}

\newcommand{\R}{\mathbb{R}}
\newcommand{\C}{\mathbb{C}}
\newcommand{\N}{\mathbb{N}}

\newcommand{\A}{\mathcal{A}}

\newcommand{\I}{\mathcal{I}}
\newcommand{\ind}{\mathfrak{i}}
\newcommand{\V}{\mathcal{V}}
\newcommand{\W}{\mathcal{W}}
\newcommand{\M}{\mathcal{M}}
\newcommand{\bigO}{\mathcal{O}}
\newcommand{\lilO}{o}
\newcommand{\Sph}{\mathbb{S}}
\newcommand{\D}{\mathcal{D}'}
\newcommand{\h}{\hbar}
\newcommand{\p}{\partial}

\newcommand{\brac}[1]{\langle #1 \rangle}
\newcommand{\pair}[2]{\bigl\langle #1,\, #2 \bigr\rangle}
\newcommand{\WF}{\operatorname{WF}'}

\newcommand{\hor}{\textnormal{(hor)}}
\newcommand{\bdf}{\rho}
\newcommand{\scf}{\textnormal{scf}}
\newcommand{\tf}{\textnormal{tf}}
\newcommand{\zf}{\textnormal{zf}}
\newcommand{\scb}{\textnormal{sc-b}}
\newcommand{\hscb}{{\hbar,\textnormal{sc-b}}}
\newcommand{\scbnorm}{{\textnormal{sc-b},h,\sigma}}
\DeclareMathOperator{\supp}{supp}
\DeclareMathOperator{\Ell}{Ell}

\DeclareMathOperator{\Diff}{Diff}
\DeclareMathOperator{\Op}{Op}
\DeclareMathOperator{\QNM}{QNM}

\usepackage[backend=biber, style=alphabetic, sorting=nyt, giveninits=true, isbn=false, url=false, doi=false, date=year, maxalphanames=5, maxbibnames=99]{biblatex}
\title[Schwarzschild QNM]{Semiclassical estimates near threshold energies and resonance counting on Schwarzschild black holes}

\author{Thomas Stucker}
\address{Department of Mathematics, ETH Zurich, R\"amistrasse 101, 8092, Z\"urich, Switzerland}
\email{thomas.stucker@math.ethz.ch}

\begin{document}

\begin{abstract}
We prove a Weyl law for the number of quasinormal modes (QNM) of a Schwarzschild black hole contained in a sector below the real axis.
This requires introducing a new pseudodifferential operator calculus tailored to the study of semiclassical spectral problems near threshold energies. Elliptic theory in this calculus can be combined with the method of complex scaling to give uniform resolvent estimates near zero energy for operators that behave at infinity like a semiclassical Schrödinger operator with a repulsive inverse-square potential. Applied to the Regge-Wheeler potential, our methods imply the absence of high angular momentum QNM from a disc whose radius grows linearly with the angular momentum. Together with the asymptotic description of Schwarzschild QNM recently obtained by Hitrik and Zworski, this shows that the number of QNM contained in a small sector below the real axis and with modulus bounded by $\lambda$ grows as $C\lambda^3$. We also study the effect of cutting off the Schwarzschild resolvent away from the event horizon and show that such a cutoff does not lead to any pole cancellations.
\end{abstract}


\maketitle

\tableofcontents

\section{Introduction}

The main goal of this paper is to prove a Weyl law for the quasinormal modes of a Schwarzschild black hole. Quasinormal modes (QNM) can be thought of as the resonances of the black hole. They occur at a discrete set of complex frequencies determined by the spacetime geometry. These correspond to exponentially damped oscillating solutions to the linear wave equation on the stationary black hole background. The study of QNM on black hole spacetimes has a long history in the physics literature, see for instance the review articles \cite{berti_cardoso,kokkotas_physicsQNM}.

Mathematically, quasinormal modes can be understood as the resonances of the spectral family $P_g(\sigma)$, that is, poles of the meromorphically continued resolvent operator $P_g(\sigma)^{-1}$. Here, $P_g(\sigma) = e^{i\sigma t}\Box_ge^{-i\sigma t}$ is the Fourier transform in time of the wave operator for the corresponding stationary metric $g$.
The theory of QNM is particularly well understood in the context of asymptotically de Sitter black holes (positive cosmological constant), see for instance \cite{dyatlov_QNM_KdS_2,dyatlov_QNM_KdS,dyatlov_QNM_quantization,vasy_KdS}. 
More recently, a rigorous characterization of QNM for the asymptotically flat family of Kerr black holes was obtained in \cite{gajic_warnick_QNM_Kerr,stucker_QNM_Kerr}.
They correspond to poles of the cutoff resolvent $\chi P_g(\sigma)^{-1}\chi$, where $\chi$ cuts off away from asymptotically flat infinity.

In the spherically symmetric Schwarzschild case, this construction was already achieved in \cite{bachelot,saBarreto_zworski} using slightly different methods.
Thanks to spherical symmetry, one can decompose the spectral family into spherical harmonics. Working with the tortoise coordinate, the resulting operator can be transformed into the spectral family of a Schr\"odinger operator on $\R$:
\begin{equation}
\label{regge-wheeler_op_intro}
    P_\ell(\sigma) = D_x^2 + V_\ell(x) - \sigma^2,
\end{equation}
where $V_\ell(x)$ is the Regge-Wheeler potential. Note that $x\to\infty$ corresponds to asymptotically flat infinity, whereas $x\to-\infty$ corresponds to the black hole horizon. The resonances of this Schr\"odinger operator provide a notion of QNM at fixed angular momentum $\ell$. 

A particularly salient feature of the Schwarzschild geometry is the presence of trapped null geodesics forming a photon sphere around the black hole. Accordingly, in the large angular momentum limit, the Schwarzschild spectral family at angular momentum $\ell$ exhibits trapping at energy $\sigma = \frac{\sqrt{\ell(\ell+1)}}{\sqrt{27}m}$, where $m$ is the mass of the black hole. This is reflected in the Regge-Wheeler potential by the presence of a unique critical point, a non-degenerate maximum. Recently, Hitrik and Zworski \cite{hitrik_zworski} gave a precise description, asymptotically in $\ell\to\infty$, of the quasinormal modes generated by the photon sphere in a full sector $\arg(\sigma) > -\theta$ below the real axis. They approach a distorted lattice in the lower half-plane. This improves on earlier results of S\'a Barreto-Zworski \cite{saBarreto_zworski} describing QNM asymptotically in strips $\Im(\sigma)>-C$. 

The asymptotic distribution of QNM, in particular, implies lower bounds on the resonance counting function: the number of QNM in the sector $\arg(\sigma) > -\theta$ with modulus $|\sigma|\leq \lambda$ grows at least as fast as $C\lambda^3$.
In the Schwarzschild-de Sitter case, where asymptotically flat infinity is replaced by a cosmological horizon, uniform resolvent estimates at the non-trapping energies show that the result of Hitrik-Zworski covers all but finitely many resonances, i.e.\ essentially all QNM are generated by the photon sphere. In this way, one can improve on the lower bound to obtain the exact order of growth of the resonance counting function, see \cite[Thm.\ 1]{hitrik_zworski}. Here, we will extend this result to the Schwarzschild case:

\begin{thm}
\label{thm_counting}
    For a small $\theta > 0$, denote by 
    \[N_\theta(\lambda) = \bigl|\QNM\cap\bigl\{|\sigma|\leq\lambda,\,\arg(\sigma)\in [-\theta,0]\bigr\}\bigr|\]
    the number of quasinormal modes -- counted with multiplicty -- of a Schwarzschild black hole of mass $m$ contained in a sector of angle $\theta$ below the positive real axis and with modulus bounded by $\lambda$. Then
    \[N_\theta(\lambda) = C\lambda^3 + \lilO(\lambda^3), \quad\text{as }\, \lambda\to\infty,\]
    where the constant $C = C_{\theta,m}$ is given by $54\sqrt{3}m^3\theta$ up to errors of order $\theta^2$.
\end{thm}

Notice that the order of growth of the counting function coincides with that given by Weyl's law for the number of $|\sigma|\leq\lambda$, counted with multiplicity, such that $\sigma^2$ is an eigenvalue of the Laplacian on a compact three-dimensional Riemannian manifold. In the context of scattering theory, similar Weyl laws for the resonance counting function have been established in certain situations. In fact, there is a rich literature on this subject. We mention the pioneering work of Melrose \cite{melrose_res_counting}, who proved polynomial upper bounds on the number of resonances for scattering by compactly supported potentials in odd-dimensional $\R^n$. This was improved by Zworski \cite{zworski_res_counting} to a sharp upper bound of order $\lambda^n$. In dimension one, Zworski also proved precise asymptotics for the resonance counting function \cite{zworski_res_counting_1D}. Vodev \cite{vodev_res_counting_general,vodev_res_counting_even} extended the $\bigO(\lambda^n)$ upper bound to even dimensions and more general compactly supported perturbations of the Laplacian. Hislop and Christiansen \cite{christiansen_hislop_res_counting_generic}  then showed that the $\lambda^n$ order of growth is attained for generic potentials of compact support.
See the monograph of Dyatlov-Zworski \cite{dyatlov_zworski_scattering_book} for further references to the literature.

The study of the asymptotic distribution of black hole quasinormal modes was initiated by S\'a Barreto-Zworski \cite{saBarreto_zworski}, who showed that the high energy QNM of Schwarzschild and Schwarzschild-de Sitter black holes are well approximated by a lattice in a strip below the real axis. A similar asymptotic description in strips of QNM for slowly rotating Kerr-de Sitter black holes was provided by Dyatlov in \cite{dyatlov_QNM_quantization}. Dyatlov also showed how an $r$-normally hyperbolic trapped set leads to a Weyl law for resonances in strips \cite{dyatlov_resonance_projector}, and applied this to slowly rotating Kerr-de Sitter black holes in \cite{dyatlov_ringdown}. More recently, progress has been made towards understanding the asymptotic distribution of QNM beyond strips. J\'ez\'equel \cite{jezequel_res_counting} proved an upper bound of order $\lambda^3$ for the number of Schwarzschild-de Sitter QNM contained in a disc of radius $\lambda$. This was complemented by lower bounds in a sector below the real axis proved by Hitrik-Zworski \cite{hitrik_zworski} for the Schwarzschild and Schwarzschild-de Sitter black holes.

The additional difficulty in establishing upper bounds for the number of resonances on Schwarzschild compared to Schwarzschild-de Sitter is related to the low energy behavior of the resolvent. In the de Sitter setting, the resolvent possesses a meromorphic continuation to the entire complex plane with a discrete set of poles. In contrast, for asymptotically flat black holes the resolvent exhibits a logarithmic singularity at zero energy. Thus, the discreteness of its poles is not automatic and there could, in principle, be infinitely many quasinormal modes accumulating at the origin. That this phenomenon does not occur for the Kerr, and hence Schwarzschild, black hole was established in \cite{stucker_QNM_Kerr}, where uniform resolvent estimates were proved in a small disc around the origin. In particular, this implies that there are no Kerr QNM of modulus $|\sigma|<\delta$ for some small $\delta>0$. Note that this result concerns poles of the full resolvent. On the other hand, Hitrik and Zworski examine the resolvent after decomposition into spherical harmonics and investigate the location of poles at a fixed angular momentum $\ell$. Their result captures all QNM at high angular momentum $\ell$, which satisfy $|\sigma|>\epsilon\ell$ for a small $\epsilon>0$, see \cite[Thm.\ 2]{hitrik_zworski}. Thus, there remains a gap between the exclusion of QNM from a disc with fixed radius and the region outside a disc of radius growing linearly with $\ell$, where the description of Hitrik-Zworski applies. We will close this gap by proving uniform estimates in $|\sigma|\leq c_0\ell$ for the resolvent of the Regge-Wheeler operator in \eqref{regge-wheeler_op_intro}.

\begin{thm}
\label{thm_exclusion}
    There exist $c_0,\ell_0,\theta_0 > 0$ such that the Schwarzschild black hole has no quasinormal modes at angular momentum $\ell\geq\ell_0$ within the region 
    \[\bigl\{\sigma\in\C\setminus\{0\},\, |\sigma|\leq c_0\ell,\, \arg(\sigma) \in [-\theta_0,0]\bigr\}.\]
    In fact, for a fixed cutoff function $\chi \in C_c^\infty(\R)$, the cutoff resolvent for the Regge-Wheeler operator is bounded uniformly in $\ell\geq\ell_0$ and $\sigma$ in the region above:
    \[\|\chi P_\ell(\sigma)^{-1}\chi u\|_{L^2(\R)} \leq C\|u\|_{L^2(\R)}, \quad \forall\, u\in L^2(\R).\]
\end{thm}

\begin{rmk}
    Note that the positivity of the Regge-Wheeler potential $V_\ell$ implies the absence of QNM in the upper half-plane, i.e.\ $\arg(\sigma)\in (0,\pi)$. Moreover, the self-adjointness of $D_x^2 + V_\ell(x)$ leads to symmetry of the cutoff resolvent, and thus QNM, with respect to the imaginary axis: $\chi P_\ell(\sigma)^{-1}\chi = \chi P_\ell(-\Bar{\sigma})^{-1}\chi$. We formulate our results for $\sigma$ in a sector below the positive real axis, but the same statements hold in a corresponding sector below the negative real axis.
\end{rmk}

Following \cite{saBarreto_zworski,hitrik_zworski}, we apply complex scaling to the operator $P_\ell(\sigma)$, see \cite[\S 7]{sjostrand_lectures_resonances} for an introduction to this method. The Regge-Wheeler potential possesses a holomorphic extension to a complex sector around the real axis, see \S \ref{section_regge-wheeler}. Restricting the extended operator to a contour $\Gamma_\beta \subset \C_z$, which coincides with $\R$ for $|z|$ in some large ball and approaches $e^{i\beta}\R$ as $|z|\to\infty$, gives a complex scaled operator
\[P_{\beta,\ell}(\sigma) = P_\ell(\sigma)|_{\Gamma_\beta}, \quad\text{with}\quad P_{\beta,\ell}(\sigma) = e^{-2i\beta}D_x^2 + V_\ell(e^{i\beta}x) - \sigma^2 \,\text{ near infinity}.\]
Note that, in contrast to \cite{stucker_QNM_Kerr}, we are using complex scaling both at asymptotically flat infinity and at the black hole horizon. Quasinormal modes can now be studied in terms of the spectrum of the complex scaled operator. 
Indeed, in $\arg(\sigma)\in(-\beta,\pi-\beta)$, the poles of the meromorphically continued cutoff resolvent correspond to the eigenvalues of $P_{\beta,\ell}(\sigma)$.

The potential grows as $\ell(\ell+1)$. We can transform the large $\ell$ behavior into a semiclassical problem by setting $h = (\ell(\ell+1))^{-\frac{1}{2}}$ and defining
\begin{equation}
\label{semiclassical_rescaing_intro}
    P_\beta(h,\sigma) = h^2P_{\beta,\ell}\Bigl(\frac{\sigma}{h}\Bigr), \quad\text{with }\, P_\beta(h,\sigma) = e^{-2i\beta}(hD_x)^2 + V(h,e^{i\beta}x) - \sigma^2 \,\text{ near infinity}.
\end{equation}
Here, $V(h,z) = h^2V_\ell(z)$ is a rescaling of the Regge-Wheeler potential.
On the real line, $V(0,x)$ has a unique non-degenerate critical point, which leads to normally hyperbolic trapping of the Hamiltonian flow of the semiclassical principal symbol at energy $\sigma = \frac{1}{\sqrt{27}m}$ corresponding to the potential maximum. This trapping generates the distorted lattice of resonances described by Hitrik-Zworski. At all other energies, the Hamiltonian flow is non-trapping. Thus, for non-zero $\sigma$ away from the trapped energy, standard semiclassical methods lead to non-trapping estimates for the complex scaled resolvent in a small sector $\arg(\sigma)>-\theta$ below the real axis and thus to the absence of spectrum for $h$ small enough, see for instance \cite[\S 12.5]{sjostrand_lectures_resonances}.
Undoing the semiclassical rescaling in \eqref{semiclassical_rescaing_intro}, this implies that for any $\delta,\epsilon>0$, we can choose $\theta>0$ small enough and $\ell_0$ large enough so that there are no quasinormal modes at angular momentum $\ell\geq\ell_0$ in
\[\bigl\{\sigma\in\C\setminus\{0\},\, \Re(\sigma)\in \bigl[\epsilon\ell,(\tfrac{1}{\sqrt{27}m}-\delta)\ell\bigr]\cup\bigl[(\tfrac{1}{\sqrt{27}m}+\delta)\ell,\infty\bigr),\, \arg(\sigma)>-\theta\bigr\},\]
see \cite[\S 2.3]{hitrik_zworski}.

The difficulty arises at the threshold energy $\sigma=0$. Because the potential decays at infinity, the semiclassical principal symbol of $P_\beta(h,0)$ is not uniformly elliptic near infinity and more sophisticated methods are needed. In the Schwarzschild-de Sitter case, where $x\to\infty$ corresponds to the cosmological horizon, the potential is exponentially decaying as $|x|\to\infty$. Thus, the approach developed by Zworski in \cite{zworski_large_parameter_scaling} is applicable. This requires using a different semiclassical parameter and a complex scaling contour that depends on $|\sigma|,\ell$. It leads to estimates for $P_{\beta,\ell}(\sigma)^{-1}$ in $K_0\leq |\sigma| \leq \epsilon\ell$, $\arg(\sigma)>-\theta$ for some $K_0>0$ and $\ell$ large enough. Together with the discreteness of the resolvent poles this provides the necessary exclusion of QNM from $|\sigma|\leq\epsilon\ell$ for large $\ell$.

In the Schwarzschild case, we use a local version of Zworski's estimates to deal with the black hole horizon, i.e. $x\to-\infty$, and introduce a new method to obtain elliptic estimates near asymptotically flat infinity, i.e. as $x\to+\infty$. This is based on the development of a new pseudodifferential calculus, which captures the transition from $|\sigma|>0$  to $|\sigma|=0$.

At asymptotically flat infinity, the analytically continued Regge-Wheeler potential decays as $z^{-2}$. Thus, the semiclassical operator in \eqref{semiclassical_rescaing_intro} behaves near infinity as
\[P_\beta(h,\sigma) \sim e^{-2i\beta}(hD_x)^2 + e^{-2i\beta}x^{-2} - \sigma^2.\]
For all fixed $|\sigma|>0$, $\arg(\sigma)\in(-\beta,\pi-\beta)$, this is uniformly elliptic as $x\to\infty$ in the usual semiclassical calculus. However, at $\sigma=0$ this differential operator should really be thought of as a semiclassical b-operator. Indeed, modulo an overall $x^{-2}$ weight,
\[P_\beta(h,0) \sim e^{-2i\beta}((hD_x)^2 + x^{-2}) = e^{-2i\beta}x^{-2}((hxD_x)^2 + ih^2xD_x + 1)\]
is uniformly elliptic in the semiclassical b-calculus, see for instance \cite[\S A.3]{hintz_vasy_nonlinear-stability} for this notion. In order to obtain uniform elliptic estimates, one should thus, in some sense, patch together the usual (or scattering) semiclassical calculus at $|\sigma|>0$ with the semiclassical b-calculus at $|\sigma|=0$.

In \S \ref{section_hscb_calculus}, we introduce a semiclasical version of the scattering-b transition calculus of \cite{guillarmou_hassell}, which is designed for this purpose. The differential operators in this calculus are built from the vector field
\[h\frac{\brac{x}}{1+|\sigma|\brac{x}}D_x,\]
which behaves like $hD_x$ for non-zero $\sigma$ and like $h\brac{x}D_x$ at $\sigma=0$. On the other hand, the coefficients of operators are conormal at infinity, i.e. enjoy regularity under repeated applications of $\brac{x}\p_x$ uniformly in $\sigma$.
Using the framework of scaled bounded geometry developed by Hintz in \cite{hintz_scaled_bounded_geometry}, we obtain a corresponding algebra of pseudodifferential operators along with a well-behaved principal symbol map. This provides microlocal elliptic estimates on certain $(|\sigma|,h)$-dependent Sobolev spaces. We will see that, near infinity, the operator $P_\beta(h,\sigma)$ above defines an elliptic element of this calculus.

The semiclassical scattering-b transition calculus can be applied more generally to semiclassical spectral problems for operators that behave near infinity like a Schr\"odinger operator with a repulsive inverse-square potential. Indeed, let $P_h$ be a second order semiclassical differential operator on $\R^n$, which is non-trapping at zero energy and approaches $h^2\Delta + \mu|x|^{-2}$ as $|x|\to\infty$ with $\mu>0$. Assume further that outside some ball $B_{R_0}$ the coefficients of $P_h$ extend to holomorphic functions in a neighborhood of $e^{i[0,\beta_0]}(R_0,\infty)\Sph^{n-1}\subset\C^n$ for some $\beta_0>0$. Then complex scaling provides a meromorphic continuation of the cutoff resolvent to the sector $\arg(\sigma) > -\beta_0$ below the positive real axis. In \S \ref{section_thresholds}, we prove uniform estimates for the cutoff resolvent of such an operator near the threshold energy $\sigma=0$, see \ref{assumption_elliptic}-\ref{assumption_asymptotics} for the precise assumptions on $P_h$. Note that we formulate our result in terms of a positive scaling angle $\beta_0$, but a corresponding estimate holds if we take $\beta_0<0$, see Remark \ref{rmk_negative_beta}.

\begin{thm}
\label{thm_threshold_estimates}
    Fix a cutoff function $\chi\in C_c^\infty(\R^n)$. There exist $h_0,\sigma_0>0$ such that for any $s\in\R$, $\delta>0$ the cutoff resolvent satisfies the estimate
    \[\|\chi(P_h-\sigma^2)^{-1}\chi f\|_{H^{s}_h(\R^n)} \leq C\|f\|_{H^{s-2}_h(\R^n)}, \quad \forall\, f\in H^{s-2}_h(\R^n),\]
    uniformly in $h\in(0,h_0]$, $\sigma\in\C\setminus\{0\}$ with $|\sigma|\leq\sigma_0$ and $\arg(\sigma)\in[-\beta_0+\delta,\pi-\delta]$.
\end{thm}

\begin{rmk}
    By modifying the pseudodifferential operator calculus of \S \ref{section_hscb_calculus}, one could prove similar uniform resolvent estimates for semiclassical operators that behave near infinity like a Schrödinger operator with a different long-range repulsive potential, i.e.\ $P_h \sim h^2\Delta + \mu|x|^{-2\gamma}$ with $\mu>0$ and $\gamma\in (0,1)$. We do not further pursue this here, but see Remark \ref{rmk_gamma_calculus} for a more extensive discussion.
\end{rmk}


Returning to our discussion of Schwarzschild quasinormal modes, it is not immediately clear that the QNM defined through the Regge-Wheeler operator agree with the QNM constructed in \cite{stucker_QNM_Kerr}. Indeed, the operator in \eqref{regge-wheeler_op_intro} lives on the exterior region $(2m,\infty)_r=\R_x$. We applied complex scaling both at infinity and at the black hole horizon and obtained quasinormal modes as the poles of $\chi P_\ell(\sigma)^{-1}\chi$, where $\chi$ cuts off to a compact subset of $(2m,\infty)_r$. On the other hand, in \cite{stucker_QNM_Kerr} we used coordinates $(t_*,r,\omega)$ for which the Schwarzschild metric extends beyond the horizon and considered the spectral family $P^\hor(\sigma) = e^{i\sigma t_*}\Box_ge^{-i\sigma t_*}$ on a spatial slice extending into the interior region. Complex scaling was only used at asymptotically flat infinity and QNM were obtained as the poles of $\chi^\hor P^\hor(\sigma)^{-1}\chi^\hor$, where $\chi^\hor$ cuts off away from infinity, but is identically one across the horizon. In \S \ref{section_qnm_comparison}, we will show that the poles of these two cutoff resolvents do in fact coincide. Thus, cutting off away from the horizon does not lead to any pole cancellations. Note that this contrasts with the case of exact de Sitter space, where \cite{hintz_xie_dual_resonances} showed that such cancellation can occur.

\begin{thm}
\label{thm_pole_cancellation}
    The resonances of the Regge-Wheeler operator $P_\ell(\sigma)$ for various $\ell$ agree with the resonances of the full spectral family $P^\hor(\sigma)$ with respect to a time-coordinate extending beyond the horizon.
\end{thm}


\subsection{Outline of the paper}

\begin{itemize}
    \item In \S \ref{section_hscb_calculus}, we introduce the semiclassical sc-b calculus. This pseudodifferential operator algebra is constructed in \S \ref{section_scaled_bounded} using the framework of scaled bounded geometry developed in \cite{hintz_scaled_bounded_geometry}. In \S \ref{section_hscb_ellipticity}, we state uniform elliptic estimates in this calculus, which are used in the rest of the paper.
    \item In \S \ref{section_thresholds}, we use the semiclassical sc-b calculus and complex scaling to prove Thm.\ \ref{thm_threshold_estimates}, establishing uniform cutoff resolvent estimates near zero energy for a class of semiclassical differential operators on $\R^n$, which behave near infinity like a Schrödinger operator with a repulsive inverse-square potential. The main ingredient is the uniform bound on the complex scaled resolvent acting between semiclassical sc-b Sobolev spaces, proved in Prop.\ \ref{prop_threshold_estimates}.
    \item In \S \ref{section_schwarzschild}, we turn to the study of quasinormal modes on a Schwarzschild black hole. In \S \ref{section_regge-wheeler}, we state requisite properties of the analytically continued Regge-Wheeler operator at angular momentum $\ell$. In \S \ref{section_infty_estimates}, we use elliptic theory in the semiclassical sc-b calculus to prove estimates on the complex scaled Regge-Wheeler operator near asymptotically flat infinity, which are uniform in $\ell\to\infty$. The semiclassical parameter is essentially the inverse of the angular momentum. In \S \ref{section_horizon_estimates}, we use the method of \cite{zworski_large_parameter_scaling} to prove similar uniform estimates near the black hole horizon. These estimates are used in \S \ref{section_weyl_law} to prove the resonance exclusion result of Thm.\ \ref{thm_exclusion}, which combined with the results of \cite{hitrik_zworski} gives the Weyl law of Thm.\ \ref{thm_counting}.
    \item In \S \ref{section_qnm_comparison}, we compare the resonances of the Regge-Wheeler operator to the poles of the cutoff resolvent for the Schwarzschild spectral family in coordinates that extend beyond the event horizon. We show that cutting the resolvent off away from the event horizon does not lead to pole cancellations, proving Thm.\ \ref{thm_pole_cancellation}. The key result is Lemma \ref{lemma_no_delta_sol}, showing that there are no dual resonant states supported on the event horizon.
\end{itemize}

\subsection*{Acknowledgements}
I wish to thank Maciej Zworski for suggesting the problem of resonance counting on Schwarzschild and for valuable discussions.
I am also grateful to Peter Hintz for many helpful suggestions and for carefully reading parts of the manuscript.

\section{Semiclassical scattering-b transition calculus}
\label{section_hscb_calculus}
In this section, we define an algebra of parameter-dependent pseudodifferential operators on $\R^n$. It is a semiclassical version of the scattering-b transition (sc-b) algebra, which was first introduced in \cite{guillarmou_hassell} and used in \cite{stucker_QNM_Kerr} to obtain uniform low energy resolvent estimates for the Kerr spectral family; see also \cite[\S A.3]{hintz_mode_stability}. We will see that this semiclassical sc-b calculus provides the right framework for studying the low energy behavior of certain semiclassical spectral problems.

The differential operators we wish to study depend both on a complex parameter $\sigma$ (the spectral parameter) and a small parameter $h$ (the semiclassical parameter). Thus, we take as our parameter space $(h,\sigma) \in P = (0,1]\times\Sigma$, where
\begin{equation}
\label{Sigma}
    \Sigma = \{\sigma \in \C,\, 0 < |\sigma| \leq c,\, \phi_- \leq \arg(\sigma) \leq \phi_+\},
\end{equation}
for some $c>0$ and $0 < \phi_+ - \phi_- < 2\pi$, is a bounded sector in the complex plane. We will be particularly interested in the limit $\sigma\to 0$, $h\to 0$.

Denote by $\rho_\scf\in C^\infty(P\times\R^n)$ the weight function
\[\rho_\scf = \frac{1}{1+|\sigma|\brac{x}},\]
where $\brac{x} = (1+|x|^2)^{\frac{1}{2}}$.
Semiclassical sc-b differential operators, $\Diff_\hscb(\R^n)$, will be built from the vector fields
\begin{equation*}
    h\rho_\scf\brac{x}\p_{x^j}, \quad j\in\{1,\dots,n\},
\end{equation*}
with parameter-dependent coefficients that are conormal at infinity, i.e.\ enjoy regularity upon repeated application of $\brac{x}\p_x$, uniformly in $\sigma,h$. That is,
\begin{equation}
\begin{split}
\label{diff_ops}
    \Diff^m_\hscb(\R^n) = \Bigl\{&\sum_{|\alpha|\leq m} a_\alpha(x,h,\sigma)\Bigr(h\frac{\brac{x}}{1+|\sigma|\brac{x}}\p_{x}\Bigr)^\alpha,\,\, \exists\,C_{\alpha,\beta}>0 \text{ such that }\\ 
    &|\p_x^\beta a_\alpha(x,h,\sigma)| \leq C_{\alpha,\beta}\brac{x}^{-|\beta|}, \,\,\forall\, (h,\sigma)\in P,\, x\in\R^n,\,\alpha,\beta\in\N_0^n\Bigr\}.
\end{split}
\end{equation}
Notice that for $\sigma=\sigma_0 \neq 0$ fixed we have $\rho_\scf \sim C\brac{x}^{-1}$, so an element $A\in\Diff_\hscb(\R^n)$ reduces upon restriction to $\sigma=\sigma_0$ to a semiclassical scattering operator $A|_{\sigma=\sigma_0}$, see \cite{vasy_zworski_hsc-calculus} for this notion. On the other hand, at $\sigma=0$ we have $\rho_\scf = 1$, so $A|_{\sigma=0}$ gives rise to a semiclassical b-operator, see for instance \cite[\S A.3]{hintz_vasy_nonlinear-stability}. The semiclassical scattering-b transition algebra can be viewed as patching together these more well-known operator algebras in a continuous manner, hence the name.

Due to the conormal regularity of the coefficients, we have
\[[\Diff^m_\hscb(\R^n),\Diff^{m'}_\hscb(\R^n)] \subset h\rho_\scf\Diff^{m+m'-1}_\hscb(\R^n).\]
Thus, by mapping $h\rho_\scf\brac{x}\p_{x^j} \to \xi^j$, the $j$-th coordinate of $(\R^n)^*$ we obtain a multiplicative principal symbol map $\sigma_\hscb$ fitting into a short exact sequence
\[0 \to h\rho_\scf\mathrm{Diff}_\hscb^{m-1}(\R^n) \hookrightarrow \mathrm{Diff}_\hscb^m(\R^n) \xrightarrow{\sigma_\hscb} \frac{P^m(P\times\R^n\times(\R^n)^*)} {h\rho_\scf P^{m-1}(P\times\R^n\times(\R^n)^*)} \to 0,\]
where $P^m(P\times\R^n\times(\R^n)^*)$ consists of functions $p\in C^\infty(P\times\R^n\times(\R^n)^*)$ that are polynomial in $(\R^n)^*$ and satisfy
\[\bigl|(\brac{x}\p_x)^\alpha(\brac{\xi}\p_\xi)^\beta p(h,\sigma,x,\xi)\bigr| \leq C_{\alpha,\beta}\brac{\xi}^m, \quad \forall\,(h,\sigma)\in P,\,x\in\R^n,\,\xi\in(\R^n)^*,\, \alpha,\beta\in\N_0^n.\]
Note that $p$ can more naturally be thought of as a function on a vector bundle $T^*_\hscb\R^n$ over $P\times\R^n$, see \S \ref{section_symbols_pseudos}.


Elements of $\Diff^m_\hscb(\R^n)$ define uniformly (in $h,\sigma$) bounded operators on corresponding semiclassical scattering-b transition Sobolev spaces $H^s_\scbnorm(\R^n)$. These are spaces of distributions on $\R^n$ equipped with $(h,\sigma)$-dependent norms. When $s$ is a positive integer the norm can be defined by
\begin{equation}
\label{sobolev_norm}
    \|u\|_{H^s_\scbnorm(\R^n)}^2 = \sum_{|\alpha|\leq s}\Bigl\|\Bigl(h\frac{\brac{x}}{1+|\sigma|\brac{x}}\p_{x}\Bigr)^\alpha u\Bigr\|_{L^2(\R^n)}^2.
\end{equation}

We will see that there exists a graded algebra of pseudodifferential operators $\Psi_\hscb(\R^n)$ microlocalizing $\mathrm{Diff}_\hscb(\R^n)$ with a well-behaved principal symbol map. Moreover, elements of $\Psi_\hscb(\R^n)$ define bounded operators between the Sobolev spaces $H_\scbnorm(\R^n)$. This leads directly to a notion of microlocal elliptic regularity for semiclassical scattering-b transition operators, which we will use to obtain uniform low energy estimates for certain semiclassical spectral problems.

Instead of defining pseudodifferential operators via their Schwartz kernel on some resolved double space, we will construct this pseudodifferential calculus using the framework of scaled bounded geometry recently introduced by Hintz \cite{hintz_scaled_bounded_geometry}.

 


\subsection{Scaled bounded geometry structure}
\label{section_scaled_bounded}
A bounded geometry (b.g.) structure on a noncompact manifold, such as $\R^n$, is given by selecting a specific collection of charts that covers the manifold. It provides a notion of uniformity as one approaches ``infinity'': objects are uniform, if they behave uniformly across the various charts in the b.g.\ structure.

A scaled b.g.\ structure additionally includes a scale factor $\rho\in(0,1]$ in each chart. This scaling specifies the basic derivatives from which to build differential operators.
Thus, a scaled b.g.\ structure determines an associated class of differential operators on the noncompact manifold: their coefficients are uniform with respect to the b.g.\ structure, whereas the fundamental frame of vector fields consists of $\rho$-rescalings of the derivatives in the charts. In fact, every scaled b.g.\ structure gives rise to an algebra of pseudodifferential operators and a scale of Sobolev spaces on which these act, see \cite{hintz_scaled_bounded_geometry}. Such operators possess a refined notion of principal symbol capturing the leading order behavior not only modulo lower order derivatives but also modulo higher order decay as $\rho\to 0$.

In this subsection, we define a scaled b.g.\ structure on $\R^n$ -- depending on the additional parameters $(h,\sigma)$ -- whose associated pseudodifferential operators are the semiclassical scattering-b transition operators alluded to above.

We cover $\R^n$ by charts $(U_\ind,\phi_\ind)$ with $\phi_\ind:U_\ind\to(-2,2)^n$ as follows:
\[U_0 = (-4,4)^n, \quad \phi_0(x) = \frac{x}{2},\]
and for $j\in \{1,\dots,n\}$, $k\in\N_0$
\begin{equation}
\label{charts}
\begin{split}
    U_{j,k,\pm} &= \Bigl\{x\in\R^n,\,\, \pm x^j \in (2^{k-2},2^{k+2}),\,\, \frac{x^i}{x^j}\in(-4,4) \,\,\forall\,i\neq j\Bigr\},\\
    \phi_{j,k,\pm}(x) &= \Bigl(\tfrac{16}{15}\bigl(\pm2^{-k}x^j\bigr)-\tfrac{34}{15},\frac{x_1}{2x_j},\dots,\frac{x_{j-1}}{2x_j},\frac{x_{j+1}}{2x_j},\dots,\frac{x_n}{2x_j}\Bigr) \in (-2,2)^n.
\end{split}
\end{equation}
Denoting by $\I = \{0\}\sqcup\bigl(\{1,\dots,n\}\times\N_0\times\{+,-\}\bigr)$ the index set of the above charts, the collection $\{(U_\ind,\phi_\ind)\}_{\ind\in\I}$ defines a bounded geometry structure on $\R^n$ in the sense of \cite[Def.\ 1.2]{hintz_scaled_bounded_geometry}. Note that this is precisely the example in \cite[\S 1.4.1]{hintz_scaled_bounded_geometry}; it can be thought of as the b.g.\ structure for the b-calculus on the radial compactification $\overline{\R^n}$.

The precise expression in \eqref{charts} is not particularly important -- it is chosen to ensure that $\phi_\ind\in(-2,2)^n$ and the $\phi_\ind^{-1}((-1,1)^n)$ cover $\R^n$ -- rather, we focus on the notion of uniform smoothness induced by the b.g.\ structure. Indeed, the b.g.\ structure is essentially determined (up to equivalence) by its space of uniformly smooth functions, see \cite[Prop.\ 2.3]{hintz_scaled_bounded_geometry}. 
This is the subspace of smooth functions on $\R^n$ whose $C^\infty$-seminorms in each coordinate chart $\phi_\ind$ are bounded uniformly in $\ind\in\I$. Note that $\brac{x}\sim |x^j| \sim 2^k$ on $U_{j,k,\pm}$ and the pullback $\phi_{j,k,\pm}^*\p_{y^i}$ of the coordinate vector fields in our elect charts can be expressed as linear combinations of the vector fields $\brac{x}\partial_{x^i}$ with coefficients that are uniform across the $U_{j,k,\pm}$. Thus, uniform smoothness with respect to the charts in \eqref{charts} corresponds to regularity upon repeated application of $\brac{x}\p_x$ and the uniformly smooth functions are given by
\[\A(\R^n) = \bigl\{f\in C^\infty(\R^n),\,\, (\brac{x}\p_x)^\alpha f \in L^\infty(\R^n), \,\,\forall\,\alpha\in\N_0^n\bigr\}.\]

Since we are interested in operators and function spaces that depend on the additional parameters $(h,\sigma)\in P$, we need to take the parameter-dependence into account in our notion of uniformity. We define a parametrized b.g.\ structure by simply using the constant b.g.\ structure $(U_\ind,\phi_\ind)$ for all values of $(h,\sigma)\in P$. This gives rise to the notion of parameter-dependent uniformly smooth functions, which lie in $\A(\R^n)$ uniformly in $(h,\sigma)$:
\begin{equation}
\label{uniform-functions}
    \A_P(\R^n) = \bigl\{f\in C^\infty(P\times\R^n),\,\, |\p_x^\alpha f(h,\sigma,x)| \leq C_\alpha\brac{x}^{-|\alpha|}, \,\,\forall\, x\in\R^n,\, (h,\sigma)\in P,\, \alpha\in\N_0^n\bigr\}.
\end{equation}

We can also define the coefficient Lie algebra $\W$, consisting of parameter-dependent vector fields on $\R^n$ whose coefficients with respect to the coordinate vector fields in the charts $(U_\ind,\phi_\ind)$ have uniformly bounded $C^\infty$-seminorms:
\begin{equation}
\label{coeff-vectors}
    \W = \Bigl\{\sum_{j=1}^n a_j(h,\sigma,x)\brac{x}\p_{x^j},\,\, a_j\in\A_P(\R^n)\,\,\forall\,j\in\{1,\dots,n\}\Bigr\}.
\end{equation}

We now introduce parameter-dependent scalings $\rho_{\ind,h,\sigma}$ on the charts $(U_\ind,\phi_\ind)$ as follows:
\begin{equation}
    \rho_{0,h,\sigma} = h, \quad \rho_{(j,k,\pm),h,\sigma} = \frac{h}{1+|\sigma|2^k}, \quad \forall\, h\in(0,1],\, \sigma\in\Sigma,\, j\in\{1,\dots,n\},\, k\in\N_0.
\end{equation}
Then $\{(U_\ind,\phi_\ind,\rho_{\ind,h,\sigma})\}_{\ind\in\I,(h,\sigma)\in P}$ defines a parameterized scaled bounded geometry structure in the sense of \cite[Def.\ 1.10]{hintz_scaled_bounded_geometry}. Note that the scalings are the same in each direction (compare \cite[Def.\ 1.4]{hintz_scaled_bounded_geometry}). Unlike the b.g.\ structure, the scale factors depend explicitly on $(h,\sigma)\in P$.

Our differential operators will be generated by (parameter-dependent) vector fields $V$ which in the charts $(U_\ind,\phi_\ind)$ take the form
\[(\phi_\ind)_*V = \sum_{j=1}^n V_\ind^j(y,h,\sigma)\rho_{\ind,h,\sigma}\partial_{y^j},\]
where the $V_\ind^j(y,h,\sigma)$ have $C^\infty$-seminorms that are uniformly bounded with respect to $\ind,h,\sigma$.
Such vector fields constitute the operator Lie algebra, which in our case amounts to
\begin{equation}
\label{op-vectors}
    \V_\hscb = \Bigl\{\sum_{j=1}^n a_j(x,h,\sigma)\frac{\brac{x}}{1+|\sigma|\brac{x}}h\p_{x^j},\,\, a_j\in \A_P(\R^n) \,\,\forall\,j\in\{1,\dots,n\}\Bigr\}.
\end{equation}
The space of $m$-th order semiclassical scattering-b transition differential operators, denoted $\Diff_\hscb^m(\R^n)$, consists of finite sums of up to $m$-fold compositions of elements of $\V_\hscb$. Notice that this gives precisely the operators introduced in \eqref{diff_ops}.

The final ingredient we need for our analysis is a notion of weight functions associated to the parametrized b.g.\ structure, see \cite[Def.\ 3.4]{hintz_scaled_bounded_geometry}. A weight is a positive function $w\in C^\infty(P\times\R^n)$ such that $w^{-1}W(w) \in \A_P(\R^n)$ for all $W\in\W$, that is,
\begin{equation}
    0 < w\in C^\infty(P\times\R^n), \quad (\brac{x}\p_x)^\alpha\Bigl(\frac{\brac{x}\p_x^jw(h,\sigma,x)}{w(h,\sigma,x)}\Bigr) \in L^\infty(P\times\R^n), \,\,\forall\, \alpha\in\N_0^n, j\in\{1,\dots,n\}.
\end{equation}
Note that if $w$ is a weight, then $w^q$ is a weight for any power $q\in\R$. Associated to our chosen scale factors is a scaling weight, which captures the scaling in the charts on a global level
\[\rho_\text{scale} = h\rho_\scf = \frac{h}{1+|\sigma|\brac{x}}.\]
Further relevant weights include powers of $h$, $|\sigma|$ and $\brac{x}$, as well as
\begin{equation}
\label{bdfs}
    \rho_\scf = \frac{1}{1+|\sigma|\brac{x}}, \quad \rho_\tf = \frac{1+|\sigma|\brac{x}}{\brac{x}}, \quad \rho_\zf = \frac{|\sigma|\brac{x}}{1+|\sigma|\brac{x}}.
\end{equation}
If $w$ is a weight, we denote by $w\Diff_\hscb^m(\R^n)$ the space of weighted semiclassical sc-b differential operators, consisting of operators $A$ such that $w^{-1}A \in \Diff_\hscb^m(\R^n)$.

\subsection{Symbols and pseudodifferential operators}
\label{section_symbols_pseudos}
In order to define symbol spaces $S^m(T^*_\hscb\R^n)$, see \cite[Def.\ 3.8]{hintz_scaled_bounded_geometry}, it is convenient to introduce the vector bundle $T^*_\hscb\R^n$ over $P\times\R^n$. Writing $y_\ind=\phi_\ind(x)$ for $x\in U_\ind$, this can be patched together from local trivializations $P\times \phi_\ind(U_\ind)\times\R^n$ in the charts using the transition functions
\[(h,\sigma,y_\ind,\xi) \in P\times \phi_\ind(U_\ind\cap U_\mathfrak{j})\times\R^n \to \Bigl(h,\sigma,y_\mathfrak{j},\sum_{j=1}^n\frac{\rho_{\mathfrak{j},h,\sigma}}{\rho_{\ind,h,\sigma}}\frac{\p y_\ind^j}{\p y_\mathfrak{j}^k}\eta_j\Bigr) \in P\times \phi_\mathfrak{j}(U_\ind\cap U_\mathfrak{j})\times\R^n.\]
Then $S^m(T^*_\hscb\R^n)$ consists of all $a\in C^\infty(T^*_\hscb\R^n)$ such that writing $a_\ind(h,\sigma,y,\xi)$ for the expression in the local trivializations above, one has
\begin{equation}
\label{symbols_loc}
    |\p_y^\alpha\p_\xi^\beta a_\ind(h,\sigma,y,\xi)| \leq C_{\alpha,\beta}\brac{\xi}^{m-|\beta|}, \quad \forall\, y\in\phi_\ind(U_\ind),\, \xi\in\R^n,\, \alpha,\beta\in\N_0^n,
\end{equation}
where the constants $C_{\alpha,\beta}$ are uniform in $\ind\in\I$.

Note that identifying $(h,\sigma,y_\ind,\xi)$ with $(h,\sigma,\frac{\xi}{\rho_{\ind,h,\sigma}}dy_\ind) \in P\times T^*U_\ind$ gives an isomorphism $T^*_\hscb\R^n \to P\times T^*\R^n$. However, this isomorphism is not uniformly bounded with respect to the b.g.\ structure. Recalling that $dy_\ind \sim \brac{x}^{-1}dx$ and $\rho_{\ind,h,\sigma} \sim h\rho_\scf$ uniformly in $\ind\in\I$, we obtain a global trivialization of $T_\hscb^*\R^n$ via the frame of one-forms $\{\frac{dx^1}{h\rho_\scf\brac{x}},\dots,\frac{dx^n}{h\rho_\scf\brac{x}}\}.$
We will write $(h,\sigma,x,\xi)\in T_\hscb^*\R^n$ for the cotangent vector $(h,\sigma,\frac{\xi}{h\rho_\scf\brac{x}})\in P\times T^*\R^n$. Taking into account the uniformity requirement in \eqref{symbols_loc}, our space of symbols can be characterized globally as
\begin{equation}
\label{symbols}
    S^m(T_\hscb^*\R^n) = \bigl\{a\in C^\infty(T_\hscb^*\R^n),\, |(\brac{x}\p_x)^\alpha\p_\xi^\beta a(h,\sigma,x,\xi)| \leq C_{\alpha,\beta}\brac{\xi}^{m-|\beta|},\, \forall\,\alpha,\beta\in\N_0^n\bigr\}.
\end{equation}

The quantization of symbols can be obtained from local quantization maps in the charts $(U_\ind,\phi_\ind)$, which take into account the scale factor. Indeed, for $a_\ind(h,\sigma,y,\xi)\in P\times (-2,2)^n\times\R^n$ symbolic in $\xi$ of order $m\in\R$ with $\supp(a_\ind) \subset P\times (-\frac{5}{4},\frac{5}{4})^n\times\R^n$ we define
\[\Op_\ind(a_\ind)u(y) = \frac{1}{(2\pi)^n}\int_{\R^n}\int_{\R^n} e^{i\frac{y-y'}{\rho_{\ind,h,\sigma}}\xi}\psi(|y-y'|)a_\ind(h,\sigma,y,\xi)u(y')\,d\xi\,\frac{dy'}{\rho_{\ind,h,\sigma}^n}, \quad \forall u\in C^\infty_c(\R^n),\]
where $\psi \in C_c^\infty((-\frac{1}{4},\frac{1}{4})^n)$ with $\psi=1$ near $0$ cuts off to a neighborhood of the diagonal. The quantization of a symbol $a\in S^m(T_\hscb^*\R^n)$ can then be patched together from these local quantization maps using a partition of unity, see \cite[Def.\ 3.31]{hintz_scaled_bounded_geometry}. In our case, this procedure is equivalent to the global quantization map
\begin{equation}
\label{quantization}
    \Op_\hscb(a)u(x) = \frac{1}{(2\pi)^n}\int_{\R^n}\int_{\R^n} e^{i\frac{x-x'}{h\rho_\scf\brac{x}}\xi}\psi(|\log\bigl(\tfrac{\brac{x'}}{\brac{x}}\bigr)|)a(h,\sigma,x,\xi)u(x')\,d\xi\,\frac{dx'}{(h\rho_\scf\brac{x})^n}.
\end{equation}

As usual, we now obtain the space of $m$-th order semiclassical sc-b pseudodifferential operators as quantizations of symbols modulo residual operators:
\[\Psi_\hscb^m(\R^n) = \bigl\{\Op_\hscb(a) + R,\,\, a\in S^m(T_\hscb^*\R^n),\, R\in h^\infty\rho_\scf^\infty\Psi_\hscb^{-\infty}(\R^n)\bigr\}.\]
Here, $h^\infty\rho_\scf^\infty\Psi_\hscb^{-\infty}(\R^n)$ is the space of residual operators, which are not only smoothing, but also increase decay as $h\to 0$ or $\rho_\scf\to 0$ by any order, see \cite[Def.\ 3.34]{hintz_scaled_bounded_geometry} for the precise definition. For any weight function $w$, we can further define the space of weighted pseudodifferential operators $w\Psi_\hscb^m(\R^n)$.

By \cite[Thm.\ 3.52]{hintz_scaled_bounded_geometry}, $\Psi_\hscb^m(\R^n)$ forms a graded algebra under composition. In fact, for any weights $w_1,w_2$ and $m_1,m_2\in\R$ we have
\[w_1\Psi_\hscb^{m_1}(\R^n) \circ w_2\Psi_\hscb^{m_2}(\R^n) \subset w_1w_2\Psi_\hscb^{m_1+m_2}(\R^n).\]
Moreover, the principal symbol map
\[\sigma_\hscb: \Op_\hscb(a) + R \in \Psi_\hscb^m(\R^n) \longrightarrow [a] \in S^m(T_\hscb^*\R^n) / h\rho_\scf S^{m-1}(T_\hscb^*\R^n),\]
defines an algebra morphism, which fits into the short exact sequence
\begin{equation}
\label{symbol_map}
    0 \to h\rho_\scf\Psi_\hscb^{m-1}(\R^n) \hookrightarrow \Psi_\hscb^m(\R^n) \xrightarrow{\sigma_\hscb} S^m(T^*_\hscb\R^n) / h\rho_\scf S^{m-1}(T^*_\hscb\R^n) \to 0.
\end{equation}

\subsection{Semiclassical scattering-b transition Sobolev spaces}
\label{section_hscb_sobolev}
Associated to the para\-met\-rized scaled b.g.\ structure are weighted Sobolev spaces $wH^s_\scbnorm(\R^n)$, on which operators in $\Psi_\hscb^m(\R^n)$ naturally act. These can be patched together from Sobolev spaces in the charts $(U_\ind,\phi_\ind)$, taking the scale factors $\rho_{\ind,h,\sigma}$ into account. Indeed, let $\chi\in C_c^\infty((-2,2)^n)$ be a cutoff function with $\chi=1$ on $(-1,1)^n$ and denote $\chi_\ind = \phi_\ind^*\chi$. Let further $S_{\ind,h,\sigma}: y\in\R^n\to\frac{y}{\rho_{\ind,h,\sigma}}\R^n$. Following \cite[Def.\ 3.14]{hintz_scaled_bounded_geometry}, we can define $wH^s_\V(\R^n)$ to consist of all families of distributions $u_{h,\sigma}\in\D(\R^n)$ such that $(\phi_\ind)_*(\chi_\ind u_{h,\sigma}) \in H^s(\R^n)$ for all $(h,\sigma) \in P$ and the parameter-dependent norm
\[\|u\|^2_{wH^s_\V(\R^n)} = \sum_{\ind\in\I}\|(S_\ind)_*(\phi_\ind)_*(\chi_\ind w^{-1}u)\|^2_{H^s(\R^n)}\]
is uniformly bounded in $(h,\sigma)\in P$. The uniform boundedness of weighted pseudodifferential operators $A\in \tilde{w}\Psi^m_\hscb(\R^n)$ as maps $A: wH^s_\V(\R^n) \to \tilde{w}wH^{s-m}_\V(\R^n)$ then follows by applying the standard boundedness result for pseudodifferential operators in the charts together with the uniform boundedness of the symbol seminorms, see \cite[Prop.\ 3.32]{hintz_scaled_bounded_geometry}.

Unlike for a general manifold with scaled b.g. structure, we have access to a natural measure on $\R^n$, i.e.\ Lebesgue measure. We will thus follow a slightly different convention from \cite{hintz_scaled_bounded_geometry} and define the unweighted semiclassical sc-b Sobolev spaces to be
\[H^s_\scbnorm(\R^n) = (h\rho_\scf\brac{x})^{\frac{n}{2}}H^s_\V(\R^n).\]
Introducing this extra weight compared to the definition above in terms of the charts ensures that the norm on $H^s_\scbnorm(\R^n)$ is equivalent to
\[\|u\|^2_{H^s_\scbnorm(\R^n)} = \|\Op_\hscb(\brac{\xi}^s)u\|^2_{L^2(\R^n)},\]
where we use the quantization map in \eqref{quantization}. In particular, $\|u\|_{H^0_\scbnorm(\R^n)} = \|u\|_{L^2(\R^n)}$ and for $s\in\N$ the norm is equivalent to \eqref{sobolev_norm}.
Of course, this redefinition does not change the boundedness, uniformly in $(h,\sigma)\in P$, of
\[A: wH^s_\scbnorm(\R^n) \to \tilde{w}wH^{s-m}_\scbnorm(\R^n), \quad \forall A\in \tilde{w}\Psi^m_\hscb(\R^n).\]
\begin{rmk}
\label{rmk_norm_equiv}
    Note that, on bounded subsets of $\R^n$, the semiclassical sc-b norm $\|\cdot\|_{H^s_\scbnorm(\R^n)}$ is equivalent to the usual semiclassical Sobolev norm $\|\cdot\|_{H^s_h(\R^n)}$. That is, for any $\chi \in C_c^\infty(\R^n)$ and any $s\in\R$ there exists $C>0$ such that
    \[\frac{1}{C}\|\chi u\|_{H^s_h(\R^n)} \leq \|\chi u\|_{H^s_\hscb(\R^n)} \leq C\|\chi u\|_{H^s_h(\R^n)}, \quad \forall\, u\in H^s_h(\R^n).\]
    Indeed, this follows from corresponding bounds $\frac{1}{C}h\leq h\frac{\brac{x}}{1+|\sigma|\brac{x}} \leq Ch$ uniformly in $|x|\leq R$, $|\sigma|\leq c$ for the weight in \eqref{sobolev_norm}.
\end{rmk}

\subsection{Ellipticity and wavefront set}
\label{section_hscb_ellipticity}
For pseudodifferential operators arising from a scaled b.g. structure, the elliptic and wavefront sets can be universally defined as subsets of the uniform compactification, see \cite[Def.\ 2.9]{hintz_scaled_bounded_geometry}. In our case, we have access to a more convenient compactification. Denote by $\overline{\R^n}$ the radial compactification of $\R^n$, that is
$$\overline{\R^n} = \R^n \sqcup \bigl([0,\infty)_\bdf\times\Sph^{n-1}_\omega\bigr) / \sim,$$
where we identify $(r,\omega) \in \R^n$, in polar coordinates, with $(\bdf,\omega) = \bigl(\frac{1}{r},\omega\bigr)$. Note that $\brac{x}^{-1}$ extends to a boundary defining function on $\overline{\R^n}$. Let further
$$\Sigma^{\text{res}} = [\Sigma;\{0\}] \simeq [0,c]_{|\sigma|}\times[\phi_-,\phi_+]_{\arg(\sigma)}$$
be the blow-up of $\Sigma$ at the point $\sigma=0$, see \eqref{Sigma}. We then define
$$X_\hscb = [0,1]_h\times X_\scb, \quad\text{where}\quad X_\scb = [\Sigma^{\text{res}}\times\overline{\R^n};(\{0\}\times[\phi_-,\phi_+])\times\partial \overline{\R^n}]$$
is the blow-up of $\Sigma^{\text{res}}\times\overline{\R^n}$ along the submanifold $(\{0\}\times[\phi_-,\phi_+])\times\partial \overline{\R^n}$, i.e.\ at $|\sigma|=0$, $\brac{x}^{-1}=0$. The manifold with corners $X_{\text{sc-b}}$ is used in \cite[\S 5.1]{stucker_QNM_Kerr} to define the scattering-b transition calculus; we refer the reader there for further details.

The weight functions in \eqref{bdfs} extend to smooth boundary defining functions on $X_\hscb$. Indeed, $X_\hscb$ has four boundary hypersurfaces that will be relevant to our analysis, see Figure \ref{fig_hscb}.
\begin{itemize}
    \item The scattering face: $\{\rho_\scf=0\} \simeq [0,1]_h\times \Sigma^{\text{res}} \times \Sph^{n-1}$,
    \item The transition face: $\{\rho_\tf=0\} \simeq [0,1]_h\times [\phi_-,\phi_+]\times [0,\infty]\times \Sph^{n-1}$,
    \item The zero face: $\{\rho_\zf=0\} \simeq [0,1]_h\times [\phi_-,\phi_+]\times \overline{\R^n}$,
    \item The semiclassical face: $\{h=0\} \simeq X_\scb$.
\end{itemize}
Using the frame $\{\frac{dx^1}{h\rho_\scf\brac{x}},\dots,\frac{dx^n}{h\rho_\scf\brac{x}}\}$ all the way down to the boundary, the semiclassical sc-b cotangent bundle extends to a vector bundle over $X_\hscb$, isomorphic to $X_\hscb\times\R^n$. Further performing a radial compactification in the fibers, we obtain a compact manifold with corners, which we denote $\overline{T_\hscb^*\R^n} \simeq X_\hscb\times\overline{\R^n}$. This constitutes an admissible compactification of $T^*_\hscb\R^n$ in the sense of \cite[Def.\ 3.62]{hintz_scaled_bounded_geometry}

\begin{figure}[t]
    \centering
    \includegraphics{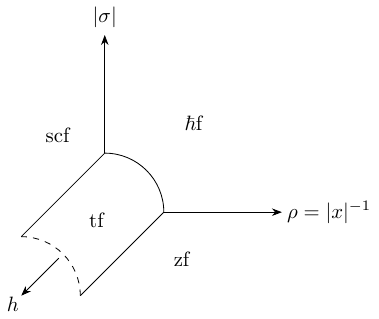}
    \caption{
    The resolved space $X_\hscb$ at fixed $\arg(\sigma)=\text{const.}$ with the various boundary faces represented. The coordinates $\omega\in\Sph^{n-1}$ are suppressed.
    }
    \label{fig_hscb}
\end{figure}

Points in the boundary of $\overline{T_\hscb^*\R^n}$ where the leading order behavior of operators is captured by the principal symbol in \eqref{symbols} form the microlocalization locus $\M$, see \cite{hintz_scaled_bounded_geometry}. Operators in $\Psi^m_\hscb(\R^n)$ are not only symbolic at fiber infinity, i.e.\ $\{\brac{\xi}^{-1}=0\}$, but also at the boundary hypersurfaces $\{\rho_\scf = 0\}$ and $\{h=0\}$ within $\overline{T_\hscb^*\R^n}$, so
\[\M = \{\brac{\xi}^{-1}=0\}\cup\{\rho_\scf = 0\}\cup\{h = 0\} \subset \partial \overline{T_\hscb^*\R^n}.\]

The elliptic set of an operator $A\in w\Psi^m_\hscb(\R^n)$ is given by those points in $\M$, where $\sigma_\hscb(A)$ is locally invertible. More precisely, denoting by $a \in wS^m(T_\hscb^*\R^n)$ a representative of $\sigma_\hscb(A)$, a point $\alpha\in\M$ lies in $\Ell_\hscb(A)$ if and only if there exists a neighborhood $\alpha \in U \subset \overline{T_\hscb^*\R^n}$ and a constant $c>0$ such that
\[\bigl|\brac{\xi}^{-m}w^{-1}a(h,\sigma,x,\xi)\bigr| \geq c, \quad\forall\, (h,\sigma,x,\xi) \in U\cap T_\hscb^*\R^n.\]
Note that this is equivalent to the existence of a symbol $b\in w^{-1}S^{-m}(T_\hscb^*\R^n)$ satisfying $ab-1 = ba-1 = 0$ on $U$.

The wavefront set of an operator $A = \Op_\hscb(a) + R \in w\Psi^m_\hscb(\R^n)$ is given by the complement of those points in $\M$ at which $a\in wS^m(T_\hscb^*\R^n)$ vanishes to infinite order. More precisely, $\alpha \in \M$ does not lie in $\WF_\hscb(a)$ if and only if there exists $\chi\in S^0(T_\hscb^*\R^n)$, elliptic at $\alpha$, such that $\chi a\in wh^N\rho_\scf^N S^{-N}(T_\hscb^*\R^n)$ for all $N\in\R$.

We then have the following microlocal elliptic regularity estimate, see \cite[Prop.\ 3.56]{hintz_scaled_bounded_geometry}:
\begin{prop}
\label{prop_hscb_elliptic_estimate}
    Let $A\in w\Psi^m_\hscb(\R^n)$ and $B\in \Psi^0_\hscb(\R^n)$ such that $\WF_\hscb(B) \subset \Ell_\hscb(A)$. Then for all $N\in\R$ there exists $C>0$ such that $\forall\, u\in \tilde{w}h^{-N}\rho_\scf^{-N}H^{-N}_\scbnorm(\R^n)$:
    \[\|Bu\|_{\tilde{w}H^s_\scbnorm(\R^n)} \leq C\bigl(\|Au\|_{w\tilde{w}H^{s-m}_\scbnorm(\R^n)} + \|u\|_{\tilde{w}h^{-N}\rho_\scf^{-N}H^{-N}_\scbnorm(\R^n)}\bigr).\]
    This holds in the strong sense that if the right hand side is finite then $Bu \in \tilde{w}H^s_\scbnorm(\R^n)$ and the estimate is satisfied.
\end{prop}


\section{Semiclassical resolvent estimates near thresholds}
\label{section_thresholds}
In this section, we exhibit how the semiclassical scattering-b transition calculus can be used to extend the usual semiclassical elliptic theory down to the threshold energy $|\sigma|=0$ for operators that behave near infinity like a Schr\"odinger operator with a repulsive inverse-square potential. More precisely, we will assume that $P_h = h^2\Delta + \mu\brac{r}^{-2} + \brac{r}^{-2}Q$, where $\mu>0$ and $Q\in\Diff_{\mathrm{b},\h}^2(\R^n)$ is a semiclassical b-operator with coefficients that decay at infinity. In order to construct the meromorphic continuation of the resolvent, we follow the complex scaling approach as expounded for instance in \cite[\S 7]{sjostrand_lectures_resonances}. We will thus assume that near infinity $P_h$ has a holomorphic extension to a an open subset in $\C^n$. We then prove uniform estimates near zero energy for the complex scaled resolvent on the semiclassical sc-b Sovolev spaces of \S \ref{section_hscb_sobolev}.

Let
\begin{equation}
\label{P_h}
    P_h = \sum_{1\leq j,k\leq n} a_{jk}(x)h^2D_{x^j}D_{x^k} + \sum_{1\leq j\leq n} b_j(h,x) hD_{x^j} + V(h,x)
\end{equation}
be a second order semiclassical differential operator, where we write $D_{x^j}=\frac{1}{i}\p_{x^j}$. Here, the $a_{jk}\in C^\infty(\R^n)$ are independent of $h$ and $b_j,V \in C^\infty([0,1]_h\times\R^n)$. The semiclassical principal symbol of $P_h$ (see e.g.\ \cite[App.\ E]{dyatlov_zworski_scattering_book} for this notion) is given by
\[\sigma_\h(P_h)(x,\xi)= \sum_{1\leq j,k\leq n} a_{jk}(x)\xi^j\xi^k + \sum_{1\leq j\leq n} b_j(0,x)\xi^j + V(0,x).\]
We assume that
\begin{enumerate}
    \item\label{assumption_elliptic} the semiclassical principal symbol is real-valued and elliptic, that is,
    \[\brac{\xi}^{-2}\sigma_\h(P_h)(x,\xi) > 0, \quad \forall\,(x,\xi)\in \overline{T}^*\R^n,\]
\end{enumerate}
where $\overline{T}^*\R^n$ is the fiber-radial compactification of the cotangent bundle.
In order to apply the method of complex scaling, we further assume that
\begin{enumerate}[resume]
    \item\label{assumption_holomorphic} there exist $\beta_0 \in (0,\pi)$, $R_0>0$ such that for each $h\in[0,1]$ the coefficients $a_{jk}(x),b_j(h,x),V(x,h)$ of $P_h$ extend from $r=|x|>R_0$ to holomorphic functions on a neighborhood $U\subset\C^n$ of
    \[\bigl\{r\omega \in \C^n\,\,|\,\, r\in\C,\, \omega\in\Sph^{n-1},\, \arg(r)\in[0,\beta_0],\, |r|>R_0\bigr\}.\]
\end{enumerate}
We also require certain conormality conditions at infinity for the analytically continued coefficients:
\begin{enumerate}[resume]
    \item\label{assumption_conormal}  for any $\alpha\in\N_0^n$
    \[(|z|\p_z)^\alpha a_{jk}(z),\quad |z|(|z|\p_z)^\alpha b_j(h,z), \quad |z|^2(|z|\p_z)^\alpha V(h,z)\]
    are bounded uniformly in $z\in U$, $h\in[0,1]$.
\end{enumerate}
Finally, we assume that the analytically continued coefficients satisfy the following asymptotics at infinity:
\begin{enumerate}[resume]
    \item\label{assumption_asymptotics} as $|z|\to\infty$ for $z\in U\subset\C^n$, we have uniformly in $U$ and $h\in[0,1]$:
    \[a_{jk}(z) \to \delta_{jk}, \quad |z|b_j(h,z) \to 0, \quad z^2V(h,z) \to \mu >0.\]
\end{enumerate}
A simple example satisfying the above assumptions is the operator
\[P_h = -h^2\Delta + \brac{x}^{-2}.\]
In \S \ref{section_schwarzschild}, we will study a more interesting example, namely the Regge-Wheeler operator for a Schwarzschild black hole, which, away from the black hole horizon, fits into the above framework.

\begin{rmk}
\label{rmk_negative_beta}
    We chose to formulate this section in terms of a positive complex scaling angle $\beta_0>0$. This leads to a meromorphic continuation of the cutoff resolvent across the positive real axis to $\arg(\sigma)< 0$. Of course, the theory applies verbatim to negative scaling angles. That is, if assumption \ref{assumption_holomorphic} on the operator $P_h$ holds for a $\beta_0<0$, then we can meromorphically continue the curoff resolvent across the negative real axis to $\arg(\sigma)>\pi$. Propositions \ref{prop_resolvent_continuation} and \ref{prop_threshold_estimates} below continue to hold for the appropriate range of $\arg(\sigma)$.
\end{rmk}
\begin{rmk}
    Note that elliptic estimates near infinity (in the semiclassical sc-b calculus) could be combined with propagation of regularity estimates away from the complex scaling region. Thus, it would be more natural to make a non-trapping assumption on our operator $P_h$, as opposed to assuming global ellipticity as in \eqref{assumption_elliptic}. However, due to our assumption on the asymptotics of the coefficients, $P_h$ has no characteristic set at zero energy outside some large ball. Thus, assuming that the bicharacteristic flow at zero energy eventually leaves any compact region is equivalent to the global ellipticity of $P_h$.
\end{rmk}

\subsection{Complex scaling and analytic continuation of the cutoff resolvent}

Due to the positivity of the principal symbol, the spectral family $P_h-\sigma^2$ is invertible for $\Im(\sigma)\gg 0$ as a bounded operator $(P_h-\sigma^2)^{-1}:H^s(\R^n)\to H^{s+2}(\R^n)$ for each $h\in(0,1]$. An analytic continuation of the resolvent to the full upper half plane is easily obtained from the analytic Fredholm theorem. Thanks to the analyticity assumption on the coefficients \ref{assumption_holomorphic}, we can use complex scaling to further continue the cutoff resolvent $\chi(P_h-\sigma^2)^{-1}\chi$ to $\arg(\sigma)>-\beta_0$. We briefly review this approach, which is essentially standard, see in particular \cite[\S 7]{sjostrand_lectures_resonances}.
\begin{prop}
\label{prop_resolvent_continuation}
    Let $P_h$ be a second order semiclassical differential operator as in \eqref{P_h} satisfying the assumptions \ref{assumption_elliptic}-\ref{assumption_asymptotics}. Then for each $h\in(0,1]$, $s\in\R$ and any $\chi\in C_c^\infty(\R^n)$ the cutoff resolvent extends from $\Im(\sigma) \gg 0$ to a meromorphic family of bounded operators with poles of finite rank
    \[\chi(P_h-\sigma^2)^{-1}\chi:H^s(\R^n)\to H^{s+2}(\R^n)\]
    for all $\sigma\in\C\setminus\{0\}$ with $\arg(\sigma)\in(-\beta_0,\pi)$.
\end{prop}

The method of complex scaling consists in deforming $\R^n$ to a family of maximally totally real submanifolds $\Gamma_\beta \subset \C^n$, which agree with $\R^n$ in $|z|<R_1$ and with $e^{i\beta}\R^n$ in $|z|>R_2$. Restricting the holomorphic extension of $P_h$ to $\Gamma_\beta$ gives a deformed operator $P_{\beta,h}$ on the real manifold $\Gamma_\beta$. One can then show that $P_{\beta,h}-\sigma^2$ is invertible as a meromorphic family of bounded operators (no cutoff necessary) for all
\begin{equation}
\label{Lambda_beta}
    \sigma\in\Lambda_\beta = \{\sigma\in\C\setminus\{0\},\, \arg(\sigma)\in(-\beta,\pi-\beta)\}.
\end{equation}
Cutting off to the ball $B_{R_1}\subset\R^n$, where no deformation was applied, one can show that $\chi(P_{\beta,h}-\sigma^2)^{-1}\chi$ is independent of the scaling angle $\beta$. This provides the analytic continuation of the cutoff resolvent, whose poles agree with the eigenvalues of the complex scaled operator $P_{\beta,h}$ in $\Lambda_\beta$.

The contour $\Gamma_\beta$ will be the image of a smooth injective map $F_\beta:\R^n\to\C^n$ of the form
\begin{equation}
\label{F_beta}
    \Gamma_\beta = F_\beta(\R^n), \quad F_\beta(x) = e^{i\phi_\beta(|x|)}x,
\end{equation}
for a smooth phase function $\phi_\beta:\R_+\to[0,\beta]$ satisfying
\[\phi_\beta(r)=0 \text{ for } r<R_1, \quad \phi_\beta(r)=\beta \text{ for } r>R_2, \quad 0\leq r\p_r\phi_\beta(r)\leq\epsilon \,\,\,\forall\, r\in\R_+.\]
Note that such a function can be constructed for any $R_1,R_2,\epsilon>0$ as long as $R_2$ is chosen large enough depending on $R_1,\epsilon$.
We define Sobolev spaces $H^s(\Gamma_\beta)$ on $\Gamma_\beta$ using the coordinate chart $F_\beta^{-1}$. Thus, $H^s(\Gamma_\beta) = (F_\beta)_*H^s(\R^n)$ is unitarily equivalent to $H^s(\R^n)$. Abusing notation, we will frequently write estimates for the complex scaled operator with respect to the norm on $H^s(\R^n)$, understanding that this should really be pulled back to the manifold $\Gamma_\beta$.

Taking $\beta\leq \beta_0$ and $R_1>R_0$, we have that $\Gamma_\beta\setminus B_{R_1}$ is contained in the open set $U\subset\C^n$, where $P_h$ defines a holomorphic differential operator by assumption \ref{assumption_holomorphic}. Thus, we can restrict $P_h$ to the totally real submanifold $\Gamma_\beta$ to obtain the complex scaled operator $P_{\beta,h} = P_h|_{\Gamma_\beta}$. Using $F_\beta^{-1}:\Gamma_\beta\to\R^n$ as a coordinate chart, we have
\begin{equation}
\begin{split}
\label{P_beta}
    P_{\beta,h} = &\sum_{1\leq j,k\leq n} a_{jk}(F_\beta(x))h^2(dF_\beta^{-T}D_{x})^j(dF_\beta^{-T}D_{x})^k \\ + &\sum_{1\leq j\leq n} b_j(h,F_\beta(x)) h(dF_\beta^{-T}D_{x})^j + V(h,F_\beta(x)),
\end{split}
\end{equation}
where $dF_\beta^{-T}$ denotes the inverse transpose of the differential of $F_\beta$.

\begin{proof}[Proof of Proposition \ref{prop_resolvent_continuation}]
    Fix $h\in (0,1]$. In $|x|<R_1$, $P_{\beta,h}$ agrees with $P_h$ and assumption \ref{assumption_elliptic} ensures that $P_h$ is elliptic (in the classical sense) in this region. In the complex scaling region $|x|>R_1$, we have
    \begin{equation}
    \label{transition_region}
        (dF_\beta(x)^{-1})^{jk} = e^{-i\phi_\beta(x)}\Bigl(\delta^{jk} - \frac{i|x|\phi_\beta'(|x|)}{1+i|x|\phi_\beta'(|x|)}\frac{x^j}{|x|}\frac{x^k}{|x|}\Bigr).
    \end{equation}
    Using $|x|\phi_\beta'(|x|) < \epsilon$, we find that
    \begin{equation*}
        (dF_\beta^{-T}\cdot\xi)^j = e^{-i\phi_\beta(x)}\xi^j + \bigO(\epsilon|\xi|),
    \end{equation*}
    and the principal symbol is given by
    \[\sigma(P_{\beta,h}) = h^2e^{-2i\phi_\beta(x)}\sum_{1\leq j,k\leq n} a_{jk}(e^{i\phi_\beta(|x|)}x)\xi^j\xi^k + \bigO(\epsilon|\xi|^2).\]
    Since $a_{jk}(z)\to\delta_{jk}$ as $|z|\to\infty$, see assumption \ref{assumption_asymptotics}, we can choose $R_1$ large enough and $\epsilon$ small enough, so that $|\sigma(P_{\beta,h})| > c|\xi|^2$ uniformly in $|x|>R_1$. Thus, $P_{\beta,h}$ is uniformly elliptic on $\R^n$. Moreover, $P_{\beta,h}-\sigma^2$ is scattering elliptic for $\sigma\in\C\setminus\{0\}$, $\arg(\sigma)\in(-\beta,\pi-\beta)$. Indeed, using that $F_\beta(x)=e^{i\beta}x$ in $|x|>R_2$ and $b_j(h,z),V(h,z) = \bigO(|z|^{-1})$, the scattering principal symbol is given near infinity by
    \[\sigma_{\mathrm{sc}}(P_{\beta,h}-\sigma^2) = h^2e^{-2i\beta}\sum_{1\leq j,k\leq n} a_{jk}(e^{i\beta}x)\xi^j\xi^k - \sigma^2 = e^{-2i\beta}\bigl(h^2|\xi|^2 - (e^{i\beta}\sigma)^2\bigr) + \lilO_{|x|\to\infty}(|\xi^2|),\]
    which is non-zero for $\arg(\sigma)\in(-\beta,\pi-\beta)$.

    Scattering elliptic theory, see for instance \cite{vasy_minicourse}, now gives the semi-Fredholm estimate
    \[\|u\|_{H^s(\R^n)} \leq C\bigl(\|(P_{\beta,h}-\sigma^2)u\|_{H^{s-2}(\R^n)} + \|u\|_{\brac{x}^{-N}H^{-N}(\R^n)}\bigr)\]
    for all $s,N\in\R$.
    Arguing similarly for the formal adjoint $(P_{\beta,h}-\sigma^2)^*$, we see that $P_{\beta,h}-\sigma^2$ defines an analytic family of Fredholm operators in $\Lambda_\beta$, see \eqref{Lambda_beta}.

    Now, the original operator $P_h=P_{0,h}$ has positive principal symbol, so $P_h-\sigma^2$ is invertible for $\Im(\sigma)$ large enough. (This follows for instance by taking $\sigma=\tilde{h}^{-1}\omega$, where $\Im(\omega)>0$ and $\tilde{h}=|\sigma|^{-1}$ is a new semiclassical parameter, and using semiclassical elliptic estimates for $\tilde{h}^2(P_h-\sigma^2)$ and its adjoint to obtain invertibility for $\tilde{h}$ small enough.) Thus, $P_{0,h}$ is a Fredholm operator of index zero and the same holds for $P_{\beta,h}$ by the stability of the Fredholm index under continuous deformations. The arguments of \cite[Prop.\ 7.5]{sjostrand_lectures_resonances}, see also \cite[Prop.\ 3.12]{stucker_QNM_Kerr}, show that $\dim(\ker(P_{\beta,h}-\sigma^2))$ is independent of $\beta$ for $\sigma\in\Lambda_\beta$. Thus, $P_{\beta,h}-\sigma^2$ is invertible for some $\sigma\in\C$ and the analytic Fredholm theorem implies that $(P_{\beta,h}-\sigma^2)^{-1}: H^s(\R^n)\to H^{s+2}(\R^n)$ exists as a meromorphic family of bounded operators on $\sigma\in\Lambda_\beta$. Finally, as in \cite[Prop.\ 3.16]{stucker_QNM_Kerr}, one can show that $\chi(P_{\beta,h}-\sigma^2)^{-1}\chi$ is independent of $\beta$ for a cutoff $\chi$ supported away from the complex scaling region. Setting $\chi(P_h-\sigma^2)^{-1}\chi = \chi(P_{\beta,h}-\sigma^2)^{-1}\chi$ thus provides a meromorphic continuation of the cutoff resolvent to $\cup_{\beta\in[0,\beta_0]}\Lambda_\beta$.
\end{proof}

\subsection{Semiclassical resolvent estimates near zero energy}
We will now show that the spectral family of the complex scaled operator, $P_{\beta,h}-\sigma^2 \in \rho_\tf^2\Diff^2_\hscb(\R^n)$, defines a weighted semiclassical scattering-b transition operator, see \ref{section_hscb_calculus}, which is moreover elliptic on all of $\R^n$ for $|\sigma|$ small enough. Elliptic regularity in the semiclassical sc-b calculus then gives the following uniform estimate for the complex scaled resolvent.
\begin{prop}
\label{prop_threshold_estimates}
    Let $P_h\in\Diff_\h^2(\R^n)$ satisfy assumptions \ref{assumption_elliptic}-\ref{assumption_asymptotics} and let $P_{\beta,h}$ be the complex scaled operator. Then for any $s\in\R$, $\delta>0$ there exist $h_0, \sigma_0 > 0$ such that the following resolvent estimate
    \begin{equation}
    \label{resolvent_estimate}
        \|(P_{\beta,h}-\sigma^2)^{-1}f\|_{H^s_\scbnorm(\R^n)} \leq C\|f\|_{\rho_\tf^2H^{s-2}_\scbnorm(\R^n)}, \quad \forall\, f\in \rho_\tf^2H^{s-2}_\scbnorm(\R^n),
    \end{equation}
    holds uniformly in $h\in(0,h_0]$, $\sigma\in\{\sigma\in\C\setminus\{0\},\, |\sigma|\leq\sigma_0,\, \arg(\sigma)\in[-\beta+\delta,\pi-\beta-\delta]\}$.
\end{prop}
\begin{proof}
The proposition will follow from the ellipticity of $P_{\beta,h}-\sigma^2$ in the semiclassical scattering-b transition calculus. We prove global elliptic estimates in this calculus: For any $s,N\in\R$ and for $\arg(\sigma)\in[-\beta+\delta,\pi-\beta-\delta]$ with $|\sigma|$ small enough, we have
\begin{equation}
\label{global_elliptic_estimate}
    \|u\|_{H^s_\scbnorm(\R^n)} \leq C\bigl(\|(P_{\beta,h}-\sigma^2)u\|_{\rho_\tf^2H^{s-2}_\scbnorm(\R^n)} + h^N\|u\|_{H^{-N}_\scbnorm(\R^n)}\bigr), \quad \forall\,u\in H^s_\scbnorm(\R^n),
\end{equation}
uniformly in $h,\sigma$.

We begin by proving a local elliptic estimate near infinity. Recall that $F_\beta(x)=e^{i\beta}x$ in $|x|>R_2$, so in this region \eqref{P_beta} becomes
\[P_{\beta,h} = e^{-2i\beta}\sum_{1\leq j,k\leq n} a_{jk}(e^{i\beta}x)h^2D_{x^j}D_{x^k} + e^{-i\beta}\sum_{1\leq j\leq n} b_j(h,e^{i\beta}x) hD_{x^j} + V(h,e^{i\beta}x).\]
Rewriting this in terms of the semiclassical sc-b vector fields $\V_\hscb^j = h\frac{\brac{x}}{1+|\sigma|\brac{x}}D_{x^j}$, we find
\begin{equation*}
\begin{split}
    P_{\beta,h} &= e^{-2i\beta}\sum_{1\leq j,k\leq n} a_{jk}(e^{i\beta}x)\Bigl(\frac{(1+|\sigma|\brac{x})^2}{\brac{x}^2}\V_\hscb^j\V_\hscb^k + ih\frac{x^j}{\brac{x}^3}\V_\hscb^k\Bigr) \\
    &+ e^{-i\beta}\sum_{1\leq j\leq n} \frac{1+|\sigma|\brac{x}}{\brac{x}}b_j(h,e^{i\beta}x) \V_\hscb^j + V(h,e^{i\beta}x).
\end{split}
\end{equation*}
Using the weights defined in \eqref{bdfs} and writing $\brac{x}^{-1} = \rho_\tf\rho_\scf$, $|\sigma| = \rho_\tf\rho_\zf$, we can factor out an overall weight of $\rho_\tf^2$ and write the spectral family of the complex scaled operator as
\begin{equation*}
\begin{split}
    \rho_\tf^{-2}(P_{\beta,h}-\sigma^2) &= e^{-2i\beta}\sum_{1\leq j,k\leq n} a_{jk}(e^{i\beta}x)\bigl(\V_\hscb^j\V_\hscb^k + ih\rho_\scf^2\frac{x^j}{\brac{x}}\V_\hscb^k\bigr) \\
    &+ e^{-i\beta}\rho_\scf\sum_{1\leq j\leq n} \brac{x}b_j(h,e^{i\beta}x) \V_\hscb^j + \rho_\scf^2\brac{x}^2V(h,e^{i\beta}x) - e^{2i\arg(\sigma)}\rho_\zf^2.
\end{split}
\end{equation*}

We see that $P_{\beta,h}-\sigma^2$ defines an element of $\rho_\tf^2\Diff^2_\hscb(\R^n)$, see \eqref{diff_ops}. The required conormality of the coefficients holds thanks to assumption \ref{assumption_conormal}. Its principal symbol in the semiclassical sc-b calculus, see \S\ref{section_symbols_pseudos}, is given by
\begin{equation*}
\begin{split}
    \sigma_\hscb\bigl(\rho_\tf^{-2}(P_{\beta,h}-\sigma^2)\bigr)(x,\xi) &= e^{-2i\beta}\sum_{1\leq j,k\leq n} a_{jk}(e^{i\beta}x)\xi^j\xi^k + e^{-i\beta}\rho_\scf\sum_{1\leq j\leq n} \brac{x}b_j(0,e^{i\beta}x) \xi^j \\
    &+ \rho_\scf^2\brac{x}^2V(0,e^{i\beta}x) - e^{2i\arg(\sigma)}\rho_\zf^2.
\end{split}
\end{equation*}

By assumption \ref{assumption_asymptotics} on the asymptotics of the coefficients, we have
\[a_{jk}(e^{i\beta}x) = \delta_{jk} + \lilO_{|x|\to\infty}(1), \quad \brac{x}b(0,e^{i\beta}x) = \lilO_{|x|\to\infty}(1), \quad \brac{x}^2V(0,e^{i\beta}x) = e^{-2i\beta}\mu + \lilO_{|x|\to\infty}(1).\]
Furthermore, by Young's inequality we can estimate
\[\Bigl|\rho_\scf\sum_{1\leq j\leq n} \brac{x}b_j(0,e^{i\beta}x) \xi^j\Bigr| \leq \sum_{1\leq j\leq n}\bigl|\brac{x}b_j(0,e^{i\beta}x)\bigr|\bigl(|\xi|^2 + \rho_\scf^2\bigr).\]
Altogether, we find
\[\sigma_\hscb\bigl(\rho_\tf^{-2}(P_{\beta,h}-\sigma^2)\bigr)(x,\xi) = e^{-2i\beta}\bigl(|\xi|^2 + \mu\rho_\scf^2\bigr) - e^{2i\arg(\sigma)}\rho_\zf^2 + \lilO_{|x|\to\infty}\bigl(|\xi|^2 + \rho_\scf^2\bigr).\]
Recalling that $\rho_\scf + \rho_\zf = 1$, we can choose $R_2$ large enough, so that
\[\bigl|\sigma_\hscb\bigl(\rho_\tf^{-2}(P_{\beta,h}-\sigma^2)\bigr)(x,\xi)\bigr| \geq c\brac{\xi}^2,\]
for some $c>0$, uniformly in $|x|>R_2$, $\xi\in\R^n$, $\arg(\sigma)\in[-\beta+\delta,\pi-\beta-\delta]$.

Thus, $P_{\beta,h}-\sigma^2 \in \rho_\tf^2\Diff^2_\hscb(\R^n)$ is semiclassically sc-b elliptic in $|x|>R_2$. Using microlocal elliptic estimates in this calculus, Proposition \ref{prop_hscb_elliptic_estimate}, we can estimate
\begin{equation}
\label{hscb_estimate_near_infty}
    \|\chi u\|_{H^s_\scbnorm(\R^n)} \leq C\bigl(\|\tilde{\chi}(P_{\beta,h}-\sigma^2)u\|_{\rho_\tf^2H^{s-2}_\scbnorm(\R^n)} + h^N\|u\|_{H^{-N}_\scbnorm(\R^n)}\bigr),
\end{equation}
where $\chi,\tilde{\chi}\in C^\infty(\R^n)$ satisfy $\supp(\chi)\subset\{|x|>R_2\}$, $\chi=1$ on $\{|x|\geq R_2+1\}$ and $\tilde{\chi}=1$ on $\supp(\chi)$. (Here, we do not care about the extra decay in $\rho_\scf$ of the error term.)

We now show that $P_{\beta,h}-\sigma^2$ is elliptic in the usual semiclassical calculus when $x$ is restricted to a bounded subset and $|\sigma|\leq\sigma_0$ is taken small enough. Indeed, away from complex scaling, i.e.\ in $|x|<R_1$, this follows from the positivity of the semiclassical principal symbol, see assumption \ref{assumption_elliptic}, if we choose $\sigma_0^2<\inf_{|x|<R_1,\xi\in\R^n}(\sigma_\h(P_h)(x,\xi))$.

Near the transition region, $|x|\in(R_1,R_2)$, the semiclassical principal symbol of $P_{\beta,h}$, see \eqref{P_beta}, is given by
\[p_{h,\beta}(x,\xi) = \sum_{1\leq j,k\leq n} a_{jk}(F_\beta(x))(dF_\beta^{-T}\xi)^j(dF_\beta^{-T}\xi)^k + \sum_{1\leq j\leq n} b_j(0,F_\beta(x))(dF_\beta^{-T}\xi)^j + V(0,F_\beta(x)).\]
Using the asymptotics of the coefficients \ref{assumption_asymptotics}, estimating the first order term by Young's inequality and writing $dF_\beta(x)^{-T} = e^{-i\phi_\beta(|x|)} + \bigO(\epsilon)$, see \eqref{transition_region}, we find
\begin{equation*}
    p_{h,\beta}(x,\xi) = e^{-2i\phi_\beta(|x|)}\bigl(|\xi|^2+\mu|x|^{-2}\bigr) + \bigO(\epsilon\xi^2) + \lilO_{|x|\to\infty}(|\xi|^2)+\lilO_{|x|\to\infty}(|x|^{-2}).
\end{equation*}
Thus, we can choose $\epsilon$ small enough and $R_1$ large enough (making $R_2$ larger if necessary) so that
\[\bigl|p_{h,\beta}(x,\xi)\bigr| \geq \frac{1}{2}\bigl(|\xi|^2 + \mu|x|^{-2}\bigr), \quad\forall\, |x|\in[R_1,R_2+2],\,\xi\in\R^n.\]
Taking $\sigma_0^2 < \frac{1}{2}\mu(R_2+2)^{-2}$, we see that $P_{\beta,h}-\sigma^2$ is semiclassically elliptic in $|x|\leq R_2+2$, uniformly in $|\sigma|\leq\sigma_0$. Semiclassical elliptic estimates, see e.g.\ \cite[Thm.\ E.33]{dyatlov_zworski_scattering_book}, now give
\begin{equation}
\label{semiclassical_estimate_bounded_x}
    \|\chi u\|_{H^s_h(\R^n)} \leq C\bigl(\|\tilde{\chi}(P_{\beta,h}-\sigma^2)u\|_{H^{s-2}_h(\R^n)} + h^N\|\tilde{\chi}u\|_{H^{-N}_h(\R^n)}\bigr),
\end{equation}
where $\supp(\chi),\supp(\tilde{\chi}) \subset \{|x|<R_2+2\}$, $\chi=1$ on $\{|x|\leq R_2+1\}$, $\tilde{\chi}=1$ on $\supp(\chi)$ and $H^s_h(\R^n)$ is the usual semiclassical Sobolev space.

On bounded domains, the usual semiclassical norm and the semiclassical sc-b norm are equivalent, see Remark \ref{rmk_norm_equiv}. Thus, combining the estimate for bounded $|x|$ in \eqref{semiclassical_estimate_bounded_x} with the estimate near infinity in \eqref{hscb_estimate_near_infty} gives the global elliptic estimate in \eqref{global_elliptic_estimate}. Choosing $h\leq h_0$ small enough, we can absorb the error term in \eqref{global_elliptic_estimate} into the left hand side and obtain the bound
\[\|u\|_{H^s_\scbnorm(\R^n)} \leq C\|(P_{\beta,h}-\sigma^2)u\|_{\rho_\tf^2H^{s-2}_\scbnorm(\R^n)}, \quad \forall\,u\in H^s_\scbnorm(\R^n),\]
uniformly in $h\in(0,h_0]$, $\sigma\in\{\sigma\in\C\setminus\{0\},\, |\sigma|\leq\sigma_0,\, \arg(\sigma)\in[-\beta+\delta,\pi-\beta-\delta]\}$. This in particular shows that the $\ker(P_{\beta,\sigma}-\sigma^2)$ is empty, which implies the invertibility of the index zero Fredholm operator $P_{\beta,\sigma}-\sigma^2$. The desired resolvent estimate follows by writing $u=(P_{\beta,\sigma}-\sigma^2)^{-1}f$.
\end{proof}

The uniform estimates for the cutoff resolvent in Thm. \ref{thm_threshold_estimates} follow immediately from the more precise estimates on the complex scaled resolvent Proposition \ref{prop_threshold_estimates}

\begin{proof}[Proof of Thm.\ \ref{thm_threshold_estimates}]
    Choosing $R_1$ large enough so that $\supp(\chi)\subset \{|x|\leq R_1\}$, the complex scaling contour $\Gamma_\beta$ coincides with $\R^n$ on the support of $\chi$ and we have 
    \[\chi(P_h-\sigma^2)^{-1}\chi = \chi(P_{\beta,h}-\sigma^2)^{-1}\chi.\]
    Recall that the semiclassical scattering-b transition norms are equivalent to the usual Sobolev norms for functions supported in a fixed bounded subset of $\R^n$, see Remark \ref{rmk_norm_equiv}. Moreover, $\rho_\tf^{-1}\leq\brac{x}\leq C<\infty$ on $\supp(\chi)$. Thus, Proposition \ref{prop_threshold_estimates} leads to
    \[\|\chi(P_h-\sigma^2)^{-1}\chi f\|_{H^s_h(\R^n)} \lesssim \|\chi(P_{\beta,h}-\sigma^2)^{-1}\chi f\|_{H^s_\scbnorm(\R^n)} \lesssim \|\chi f\|_{\rho_\tf^2H^{s-2}_\scbnorm(\R^n)} \lesssim \|f\|_{H^{s-2}_h(\R^n)}\]
    uniformly in $h\leq h_0$, $|\sigma|\leq\sigma_0$, $\arg(\sigma)\in[-\beta+\delta,\pi-\beta+\delta]$. Since this holds for all $\beta\in[0,\beta_0]$, we obtain the uniform estimate in the theorem.
\end{proof}

\begin{rmk}
\label{rmk_gamma_calculus}
    As we saw above, the repulsive inverse-square nature of the potential term $V(h,x)$ leads to uniform ellipticity near infinity of $P_h-\sigma^2$ in the semiclassical scattering-b transition calculus, which in turn implies the uniform resolvent estimate of Thm.\ \ref{thm_threshold_estimates}. By modifying the pseudodifferential operator calculus of \S \ref{section_hscb_calculus}, one can use a similar strategy to deal with operators that behave near infinity like a semiclassical Schrödinger operator with a different repulsive long-range potential. Indeed, assume that $P_h \sim h^2\Delta + \mu\brac{x}^{-2\gamma}$ near infinity with $\mu>0$ and $\gamma\in (0,1)$. Using the scaled bounded geometry perspective as in \S \ref{section_hscb_calculus}, we can define a calculus whose differential operators are built from the vector fields
    \[\V^j = h\frac{\brac{x}^\gamma}{1+|\sigma|\brac{x}^\gamma}D_{x^j}.\]
    Note that $\V^j$ are semiclassical scattering vector fields at fixed $|\sigma|>0$, whereas at $|\sigma|=0$ we have $\V^j = h\brac{x}^\gamma D_{x^j}$.
    Modifying the weight functions in \eqref{bdfs} to
    \begin{equation*}
        \rho_\scf = \frac{1}{1+|\sigma|\brac{x}^\gamma}, \quad \rho_\tf = \frac{1+|\sigma|\brac{x}^\gamma}{\brac{x}^\gamma}, \quad \rho_\zf = \frac{|\sigma|\brac{x}^\gamma}{1+|\sigma|\brac{x}^\gamma},
    \end{equation*}
    we find near infinity:
    \begin{equation}
    \label{gamma_operator}
        P_h-\sigma^2 \sim \rho_\tf^2\Bigl(\sum_{j=1}^n\bigl(\V^j\V^j + ih\gamma\frac{1}{\brac{x}^{1-\gamma}}\rho_\scf^2\frac{x^j}{\brac{x}}\V^j\bigr) + \mu\rho_\scf^2 - e^{2i\arg(\sigma)}\rho_\zf^2\Bigr).
    \end{equation}
    Applying complex scaling, the principal symbol of the complex scaled operator in this calculus can be written near infinity as
    \[\sigma_\h\bigl(\rho_\tf^{-2}(P_{\beta,h}-\sigma^2)\bigr)(x,\xi) \sim e^{-2i\beta}|\xi|^2 + e^{-2i\gamma\beta}\mu\rho_\scf^2 - e^{2i\arg(\sigma)}\rho_\zf^2.\]
    This is uniformly elliptic for $\arg(\sigma)$ in a compact subset of $(-\gamma\beta,\pi-\beta)$. We thus obtain a uniform resolvent estimate as in Thm.\ \ref{thm_threshold_estimates}, but with $\arg(\sigma)$ restricted to $[-\gamma\beta_0+\delta,\pi-\delta]$ for some $\delta>0$. Note that this only works for the long-range case $\gamma\leq 1$, since otherwise the second term in \eqref{gamma_operator} ceases to be subleading.
\end{rmk}

\section{Quasinormal modes for the Schwarzschild black hole}
\label{section_schwarzschild}

Consider the Schwarzschild metric on the region exterior to the black hole horizon, which can be written in spherical coordinates as
\[M = \R_t \times (2m,\infty)_r \times \Sph^2_\omega, \quad g = \mu(r)dt^2 - \mu(r)^{-1}dr^2 - r^2|d\omega|^2, \quad \mu(r) = 1-\frac{2m}{r},\]
where $\omega$ denotes coordinates on $\Sph^2$. Note that the metric extends smoothly across the black hole horizon at $2m$, which can be seen by a judicious choice of coordinates. However, with the exception of \S \ref{section_qnm_comparison}, we will work only in the exterior region.

The Laplace-Beltrami operator for this metric takes the form
\[\Box_g = \mu(r)^{-1}\p_t^2 - r^{-2}\p_rr^2\mu(r)\p_r - r^{-2}\Delta_\omega,\]
where $\Delta_\omega$ denotes the Laplacian on $\Sph^2$. Taking the Fourier transform in time gives the Schwarzschild spectral family
\begin{equation}
\label{spectral_family_t}
    P(\sigma) = e^{i\sigma t}\,\Box_g \,e^{-i\sigma t} = - r^{-2}\p_rr^2\mu(r)\p_r - r^{-2}\Delta_\omega - \mu(r)^{-1}\sigma^2
\end{equation}
Notice that this operator enjoys rotational symmetry. Thus, we can decompose into spherical harmonics $\{Y_{\ell,m}\}_{\ell,m}$, leading to an operator on $(2m,\infty)_r$. We further multiply through by $\mu(r)>0$ and conjugate by $r$ to bring this operator into a more convenient form:
\begin{equation}
\label{spectral_family_t_modification}
    P_\ell(\sigma) = r\mu(r)P(\sigma)\bigr|_{Y_{\ell,m}}r^{-1} = \mu(r)D_r\mu(r)D_r + r^{-2}\mu(r)\ell(\ell+1) + r^{-1}\mu(r)(\p_r\mu(r)) - \sigma^2.
\end{equation}
Making the change of variables
\begin{equation}
\label{x(r)}
    x:(2m,\infty)\to\R ,\quad x(r) = r + 2m\log(r-2m), \quad\text{with}\quad x'(r) = \mu(r)^{-1},
\end{equation}
we end up with the spectral family of a Schr\"odinger operator on $\R$:
\begin{equation}
\label{regge-wheeler}
    P_\ell(\sigma) = D_x^2 + V_\ell(x) - \sigma^2, \quad\text{where}\quad V_\ell(x) = \Bigl(1-\frac{2m}{r(x)}\Bigr)\Bigl(\frac{\ell(\ell+1)}{r(x)^2} + \frac{2m}{r(x)^3}\Bigr)
\end{equation}
is known as the Regge-Wheeler potential. Note that $x(r)$ in \eqref{x(r)} maps $(2m,\infty)$ onto $\R$ bijectively.

\subsection{Analytically continued Regge-Wheeler potential}
\label{section_regge-wheeler}

We will follow \cite{saBarreto_zworski,hitrik_zworski} and apply complex scaling to the operator $P_\ell(\sigma)$. This is possible since the potential $V_\ell(x)$ extends analytically to a domain in the complex plane; which in turn follows from the analytic continuation of the change of coordinate map $r(x)$. We thus collect some facts regarding the extension of $r(x)$ and $V_\ell(x)$ as holomorphic functions, which for the most part are contained in \cite{bachelot,hitrik_zworski}.

\begin{lemma}
\label{lemma_coordinates}
    Denote by $r:x\in\R\to r(x)\in(2m,\infty)$ the inverse of the change of coordinates in \eqref{x(r)}. Then, taking $A>0$ large enough and $\delta>0$ small enough, $r$ extends to a holomorphic function with $r(z)\neq 0$ on the domain
    \begin{equation}
    \label{domain}
        \Omega = \{z\in\C,\, |\Re(z)| > A\}\cup\{z\in\C,\, |\Im(z)| < \delta\}.
    \end{equation}
    Moreover, $r(z)$ maps $\{z\in\C,\, \Re(z) > A\}$ injectively into $\C\setminus\R_{\leq 2m}$ and satisfies \[z = r(z) + 2m\log(r(z)-2m), \quad\forall\,z\in\C,\, \Re(z) > A.\]
\end{lemma}
\begin{proof}
Note that the map $x(r)$ in \eqref{x(r)} is real analytic and bijective, so its inverse $r(x)$ defines an analytic function on $\R$, and thus extends to a holomorphic function $r(z)$ in a complex neighborhood of $\R$. Since $r(x)>2m$ for $x\in\R$, this neighborhood can be chosen so that $r\neq 0$.

Using the Lagrange inversion formula, it can further be shown that $r(z)$ is given for $z\in\R$, $z\ll 0$ by the series
\begin{equation}
\label{small_x_expansion}
    r(z) = 2m + 2m\sum_{n=1}^\infty (-1)^{n+1}\frac{n^{n-1}}{n!}e^{n(\frac{z}{2m}-1)},
\end{equation}
see \cite[Prop.\ IV.2]{bachelot}. Choosing $A>0$ large enough, this converges uniformly for all $z\in\C$, $\Re(z)<-A$ and defines a holomorphic extension to this domain. We can moreover choose $A$ so that $|r(z)-2m|<2m$, i.e.\ $r(z)\neq 0$. 

Again by \cite[Prop.\ IV.2]{bachelot}, for $z\in\R$, $z\gg 0$, the map $r(z)$ can be represented by a series
\begin{equation}
\label{large_x_expansion}
    r(z) = z - 2m\log(z) + \sum_{n=1}^\infty\sum_{m=0}^\infty c_{n,m} \frac{\log(z)^m}{z^n}, \quad\text{for some}\quad c_{n,m}\in\R,
\end{equation}
which, choosing $A>0$ large enough, converges uniformly for all $z\in\C$, $\Re(z)>A$, see also \cite[\S 2.1]{hitrik_zworski}. Together with the previous observations, this gives the holomorphic extension to the domain in \eqref{domain}.

To see the last part of the lemma, note that for $\Re(z)>A$ we have $r(z)\in\C\setminus\R_{\leq 2m}$. Indeed, choosing $A$ so that the double sum in \eqref{large_x_expansion} is bounded in absolute value by $2m$ on $\Re(z)>A$, we have 
\[|\Im(r(z))| \geq |\Im(z)| - 2m\pi - 2m, \quad \Re(r(z)) \geq \Re(z) - 2m\log(|z|) - 2m, \quad\forall\,\Re(z)>A.\]
If $r(z)\in(-\infty,2m)$, the first equation implies $|z| \leq \Re(z) + 2m(\pi+1)$, and from the second we find $\Re(z) \leq 2m\log\bigl(\Re(z)+2m(\pi+1)\bigr) + 4m$, which is impossible for $\Re(z)$ large enough. Thus, $r(z)$ is contained in the domain of the holomorphic function $r\to r+2m\log(r-2m)$. Since $r(z)+2m\log(r(z)-2m)=z$ for $z\in\R$, this equation extends to all of $\Re(z)>A$, which also shows the injectivity of $r(z)$ on this domain.
\end{proof}

\begin{lemma}
\label{lemma_potential}
    The potential $V_\ell(x)$ in \eqref{regge-wheeler} extends to a holomorphic function
    \[V_\ell:\Omega\to\C\]
    on the domain given in \eqref{domain}. In $\Re(z)\geq 0$, it decays like $z^{-2}$. More precisely, for all $k\in\N_0$ we have
    \begin{equation}
    \label{behavior_large_x}
        \sup_{\substack{z\in\Omega \\ \Re(z)\geq 0}}\bigl|(z\partial_z)^kz^2V_\ell(z)\bigr| \leq C_k\brac{\ell}^2, \quad \inf_{\substack{z\in\Omega \\ \Re(z)\geq 0}}\bigl|z^2V_\ell(z)\bigr| \geq c\brac{\ell}^2, \quad \lim_{\Re(z)\to\infty}z^2V_\ell(z) = \ell(\ell+1).
    \end{equation}
    In $\Re(z)\leq 0$, the potential is exponentially decaying. More precisely, for all $k\in\N_0$ we have
    \begin{equation}
    \begin{gathered}
    \label{behavior_small_x}
        \sup_{\substack{z\in\Omega \\ \Re(z)\leq 0}}\bigl|\partial_z^k\bigl(e^{-\frac{z}{2m}}V_\ell(z)\bigr)\bigr| \leq C_k\brac{\ell}^2, \quad \inf_{\substack{z\in\Omega \\ \Re(z)\leq 0}}\bigl|e^{-\frac{z}{2m}}V_\ell(z)\bigr| \geq c\brac{\ell}^2, \\
        \text{and}\quad \lim_{\Re(z)\to -\infty}e^{-\frac{z}{2m}}V_\ell(z) = \frac{\ell(\ell+1)+1}{e(2m)^2}.
    \end{gathered}
    \end{equation}
\end{lemma}
\begin{proof}
    For $z\in\R$, we have
    \begin{equation}
    \label{regge-wheeler_2}
        V_\ell(z) = W_\ell(r(z)), \quad\text{where}\quad W_\ell(r) = \Bigl(1-\frac{2m}{r}\Bigr)\Bigl(\frac{\ell(\ell+1)}{r^2}+\frac{2m}{r^3}\Bigr).
    \end{equation}
    Notice that $W_\ell$ defines a holomorphic function on $\C\setminus\{0\}$. Since $r(z)$ extends to a holomorphic function $r:\Omega\to\C\setminus\{0\}$ by Lemma \ref{lemma_coordinates}, the potential $V_\ell(z)=W_\ell(r(z))$ extends analytically to the same domain.

    For $z$ in a bounded subset of $\Omega$, the estimates on the derivatives of $z^2V_\ell(z)$, respectively $e^{\frac{z}{2m}}V_\ell(z)$, follow from holomorphicity and the form of the potential \eqref{regge-wheeler_2}. Moreover, for any $A>0$ we have $x\in(-A,A) \to r(x)\in(2m+\epsilon,C)$ for some $\epsilon,C>0$. Choosing $\delta>0$ small enough, so that $r(z)$ is contained in a small complex neighborhood of $(2m+\epsilon,C)$ for all $z$ in $\{z\in\C,\,|\Re(z)|<A,\, \Im(z)<\delta\}$, the lower bound on $z^2V_\ell(z)$, respectively $e^{\frac{z}{2m}}V_\ell(z)$, follows from \eqref{regge-wheeler_2}.

    It remains to consider $\Re(z) > A$ and $\Re(z) < -A$.
    By Lemma \ref{lemma_coordinates}, $z\in\{\Re(z)>A\}\to r(z)\in \C\setminus\{0\}$ is a biholomorphism onto its image with inverse $z(r) = r+2m\log(r-2m)$. Thus, for $\Re(z)>A$ we have
    \begin{equation}
    \label{potential_large_x}
        z^2V_\ell(z) = \bigl(r+2m\log(r-2m)\bigr)^2W_\ell(r) = \Bigl(1+2m\frac{\log(r-2m)}{r}\Bigr)^2\Bigl(1-\frac{2m}{r}\Bigr)\Bigl(\ell(\ell+1) + \frac{2m}{r}\Bigr)
    \end{equation}
    Since $\Re(z) = \Re(r) + 2m\log(|r-2m|) \leq (2m+1)|r| + (2m)^2$, we find that $\Re(z)\to\infty$ implies $|r|\to\infty$. Now \eqref{potential_large_x} shows immediately that $z^2V_\ell(z)\to\ell(\ell+1)$ as $\Re(z)\to\infty$. Moreover, it is evident from \eqref{potential_large_x} that
    \[\bigl|(r\p_r)^k (r+2m\log(r-2m))^2W_\ell(r)\bigr| \leq C_k\brac{\ell}^2\]
    for all $k\in\N_0$ and all $|r|>B$ with $B>0$ sufficiently large. In $\Re(z)>A$ we can use the inverse function theorem to write
    \[z\p_z = f(r)r\p_r, \quad\text{where}\quad f(r) = \Bigl(1+2m\frac{\log(r-2m)}{r}\Bigr)\Bigl(1-\frac{2m}{r}\Bigr)\]
    satisfies $|(r\p_r)^kf(r)| \leq C_k$ for $k\in\N_0$ and $|r|$ large. Choosing $A$ large, we have $|r|>B$ when $\Re(z)>A$ and we find
    \[\sup_{\Re(z)>A}\bigl|(z\partial_z)^kz^2V_\ell(z)\bigr| \leq \sup_{|r|>B}\bigl|(f(r)r\p_r)^k(r+2m\log(r-2m))^2W_\ell(r)\bigr| \leq C_k\brac{\ell}^2\]
    for all $k\in\N_0$ and some $C_k > 0$.

    In $\Re(z)<-A$, the series expansion of $r(z)$ in \eqref{small_x_expansion} shows that
    \[r(z)-2m = e^{\frac{z}{2m}}h(z), \quad\text{where}\quad h(z) = \frac{2m}{e} + \sum_{n=1}^\infty (-1)^n\frac{(n+1)^n}{(n+1)!e^{n+1}}e^{n\frac{z}{2m}}.\]
    Note that both $r(z)$ and $h(z)$ have bounded derivatives of any order in $\Re(z)<-A$ and
    \[r(z) \to 2m, \quad h(z) \to \frac{2m}{e}, \quad\text{as }\, \Re(z)\to -\infty.\]
    Using \eqref{regge-wheeler_2}, we see that
    \[e^{-\frac{z}{2m}}V_\ell(z) = h(z)\Bigl(\frac{\ell(\ell+1)}{r(z)^3} + \frac{2m}{r(z)^4}\Bigr)\]
    in $\Re(z)<-A$, so \eqref{behavior_small_x} follows from the above observations.
\end{proof}

Following \cite{saBarreto_zworski,hitrik_zworski}, we can now apply complex scaling to the Regge-Wheeler potential and obtain the Schwarzschild quasinormal modes at angular momentum $\ell$ in terms of the spectrum of the non-self-adjoint operator $(D_z^2 + V_\ell(z))|_{\Gamma_\beta}$. Here, $\Gamma_\beta$ is a contour contained in the domain $\Omega$, see \eqref{domain}, which coincides with $e^{i\beta}\R$ outside some compact set. Note that, taking $\Gamma_\beta$ as in \eqref{F_beta} with $R_1 > A$, we have $\Gamma_\beta \subset\Omega$ for all $\beta\in(-\frac{\pi}{2},\frac{\pi}{2})$. Just as in the proof of Prop.\ \ref{prop_resolvent_continuation} one can show that the complex scaled spectral family $P_{\beta,\ell}(\sigma) = P_\ell(\sigma)|_{\Gamma_\beta}$ has a meromorphic inverse in $\Lambda_\beta = \{\sigma\in\C\setminus\{0\},\, \arg(\sigma)\in(-\beta,\pi-\beta)\}$ and these can be patched up to provide the meromorphic continuation of the cutoff resolvent $\chi P_\ell(\sigma)^{-1}\chi$ to $\sigma\in\C\setminus i\R_{\leq 0}$. The quasinormal modes at angular momentum $\ell$ are the poles of $P_{\beta,\ell}(\sigma)^{-1}$ in $\Lambda_\beta$. In the following, we will prove the absence of such QNM in a disc around zero energy with a radius growing linearly in $\ell$. In \S \ref{section_infty_estimates}, we study $P_{\beta,\ell}(\sigma)$ in $x>0$ and prove elliptic estimates in the semiclassical scattering-b transition calculus. In \S \ref{section_horizon_estimates}, we turn to $x<0$ and use the method of \cite{zworski_large_parameter_scaling} to obtain estimates for $P_{\beta,\ell}(\sigma)$ in this region. These estimates can then be glued together to show the invertibility of $P_{\beta,\ell}(\sigma)$ for $|\sigma|\leq c_0\ell$ and $\ell$ large.

\begin{rmk}
\label{rmk_simpler_contours}
    Note from Lemma \ref{lemma_potential}, that $V_\ell$ is in fact holomorphic in a neighborhood of $\cup_{\beta\in[0,\beta_0]}e^{i\beta}\R$ for small enough $\beta_0$. Thus, one could also apply complex scaling to the Regge-Wheeler operator using the simpler contours $\Gamma_\beta = e^{i\beta}\R$. However, these contours are not suitable for the exclusion of low energy modes, see \S\ref{section_horizon_estimates}.
\end{rmk}

\subsection{Estimates near asymptotically flat infinity}
\label{section_infty_estimates}
In this section, we use the semiclassical scattering-b transition calculus of \S\ref{section_hscb_calculus} to prove elliptic estimates for the complex scaled Regge-Wheeler operator $P_{\beta,\ell}(\sigma)$ in the region $\Gamma_\beta\cap\{\Re(z)>-1\} = F_\beta((-1,\infty)$. Here, $\Gamma_\beta = F_\beta(\R)$ is a contour as in \eqref{F_beta} with $\beta\in(0,\frac{\pi}{2})$, i.e.\
\begin{equation}
\label{contour_infty}
    F_\beta:\R\to\C, \quad F_\beta(x) = e^{i\phi_\beta(x)}x, \quad \phi_\beta(x) = \begin{cases} 0, \quad x\in [-1,R_1], \\ \beta, \quad x\in [R_2,\infty),\end{cases}\,\, x\phi_\beta'(x) \in [0,\epsilon] \,\,\,\forall\, x.
\end{equation}
Of course, we take $R_1 > A$ so that $\Gamma_\beta$ is contained in the domain of holomorphicity of $V_\ell$, see \eqref{domain}. Since we only consider $x\in (-1,\infty)$, the precise nature of $F_\beta(x)$ for $x\in (-\infty,-1)$ is immaterial.

We introduce the semiclassical parameter $h^2=\ell(\ell+1)$. Continuing to denote by $P_\ell(\sigma)$ the analytic continuation of the operator in \eqref{regge-wheeler} to $\Omega\subset\C$, we consider the semiclassical rescaling
\begin{equation}
\label{semiclassical_rescaling}
\begin{split}
    P(h,\sigma) &= h^2P_\ell\Bigl(\frac{\sigma}{h}\Bigr) = (hD_z)^2 + V(h,z) - \sigma^2, \quad h = \bigl(\ell(\ell+1)\bigr)^{-\frac{1}{2}}, \\ 
    V(h,z) &= h^2V_\ell(z) = \Bigl(1-\frac{2m}{r(z)}\Bigr)\Bigl(\frac{1}{r(z)^2} + h^2\frac{2m}{r(z)^3}\Bigr)
\end{split}
\end{equation}
By Lemma \ref{lemma_potential}, in $\Re(z)>-1$ the operator $P(h,\sigma)$ is precisely of the form studied in \S \ref{section_thresholds}. Thus, we can obtain semiclassical elliptic estimates -- locally in $F_\beta((-1,\infty))$ -- for the complex scaled operator $P_\beta(h,\sigma)$ as in the proof of Prop.\ \ref{prop_threshold_estimates}. Undoing the semiclassical rescaling in \eqref{semiclassical_rescaling}, these take the following form for $P_{\beta,\ell}(\sigma)$, where we identify $H^s(\Gamma_\beta)$ with $H^s(\R)$ via the pullback by the map $F_\beta$, i.e.\ the Sobolev norms are with respect to local coordinates in the chart $F_\beta^{-1}$.

\begin{lemma}
\label{lemma_infty_estimates}
    Let $\chi_+,\tilde{\chi}_+\in C^\infty(\R)$ satisfy 
    \[\supp(\chi_+),\,\supp(\tilde{\chi}_+) \subset (-1,\infty),\quad \chi_+=1 \,\text{ on }\, (-\tfrac{1}{2},\infty),\quad \tilde{\chi}_+ = 1 \,\text{ on }\, \supp(\chi_+).\] 
    There exists $c_0>0$ such that for all $s,N\in\R$ the following estimate holds:
    \begin{equation}
    \label{estimate_infty_side}
        \|\chi_+ u\|_{H^s_{\scb,\brac{\ell}^{-1},\sigma}(\R)} \leq C_{s,N}\bigl(\brac{\ell}^{-2}\|\tilde{\chi}_+P_{\beta,\ell}(\sigma)u\|_{\rho_\tf^2H^{s-2}_{\scb,\brac{\ell}^{-1},\sigma}(\R)} + \brac{\ell}^{-N}\|\tilde{\chi}_+u\|_{H^{-N}_{\scb,\brac{\ell}^{-1},\sigma}(\R)}\bigr),
    \end{equation}
    uniformly in $\ell\in\N_0$ and $\sigma\in\{\sigma\in\C\setminus\{0\},\, |\sigma|\leq c_0\sqrt{\ell(\ell+1)},\, \arg(\sigma)\in[-\beta+\delta,\pi-\beta-\delta]\}$.
\end{lemma}

\begin{rmk}
\label{rmk_exotic_norms}
    See \eqref{sobolev_norm}, with $h$ replaced by $\brac{\ell}^{-1}$, for the somewhat exotic norms appearing in \eqref{estimate_infty_side} (recall also that $\rho_\tf = \frac{1+|\sigma|\brac{x}}{\brac{x}}$). This estimate in particular implies that
    \[\|\chi_+ u\|_{L^2(\R)} \leq C\bigl(\brac{\ell}^{-2}\|\rho_\tf^{-2}\tilde{\chi}_+P_{\beta,\ell}(\sigma)u\|_{L^2(\R)} + \brac{\ell}^{-N}\|\tilde{\chi}_+u\|_{L^2(\R)}\bigr),\]
    which, using $\rho_\tf^{-1}\leq \max(\brac{x},|\sigma|^{-1})$, in turn implies both of the following uniform estimates
    \begin{equation*}
    \begin{split}
        \|\chi_+ u\|_{L^2(\R)} &\leq C\bigl(\brac{\ell}^{-2}\|\brac{x}^2\tilde{\chi}_+P_{\beta,\ell}(\sigma)u\|_{L^2(\R)} + \brac{\ell}^{-N}\|\tilde{\chi}_+u\|_{L^2(\R)}\bigr), \\
        \|\chi_+ u\|_{L^2(\R)} &\leq C\bigl(|\sigma|^{-2}\brac{\ell}^{-2}\|\tilde{\chi}_+P_{\beta,\ell}(\sigma)u\|_{L^2(\R)} + \brac{\ell}^{-N}\|\tilde{\chi}_+u\|_{L^2(\R)}\bigr).
    \end{split}
    \end{equation*}
\end{rmk}
\begin{proof}
    Let $P_\beta(h,\sigma) = P(h,\sigma)|_{\Gamma_\beta} = h^2P_{\beta,\ell}(\frac{\sigma}{h})$ with $h = (\ell(\ell+1))^{-\frac{1}{2}}$, see \eqref{semiclassical_rescaling}.
    Identifying $\Gamma_\beta$ with $\R$ via the map $x\in\R \to e^{i\phi_\beta(x)}x \in \Gamma_\beta$, we have
    \[P_\beta(h,\sigma) = \Bigl(\frac{e^{-i\phi_\beta(x)}}{1+ix\phi_\beta'(x)}hD_x\Bigr)^2 + V(h,e^{i\phi_\beta(x)}x) - \sigma^2.\]
    By Lemma \ref{lemma_potential}, the potential satisfies
    \begin{equation*}
        (x\p_x)^k(x^2V(h,e^{i\phi_\beta(x)}x)) \leq C_k,\quad \forall\,x>0,\,k\in\N_0, \quad\text{and}\quad \lim_{x\to\infty}x^2V(h,e^{i\phi_\beta(x)}x) = e^{-2i\beta}.
    \end{equation*}
    Thus, cutting off to $x\in (-1,\infty)$ we see that $\tilde{\chi}_+ P_\beta(h,\sigma) \in \rho_\tf^2\Diff_\hscb^2(\R)$. We proceed as in the proof of Prop.\ \ref{prop_threshold_estimates}. In $x\geq R_2$, we have $\phi_\beta(x)=\beta$ and
    \begin{equation*}
        \rho_\tf^{-2}P_\beta(h,\sigma) = e^{-2i\beta}\bigl(V_\hscb^2 + ih\rho_\scf^2\frac{x}{\brac{x}}V_\hscb\bigr) + \rho_\scf^2\brac{x}^2V(h,e^{i\phi_\beta(x)}x) - e^{2i\arg(\sigma)}\rho_\zf^2,
    \end{equation*}
    where
    \[V_\hscb = \frac{\brac{x}}{1+|\sigma|\brac{x}}hD_x, \quad \rho_\tf = \frac{1+|\sigma|\brac{x}}{\brac{x}}, \quad \rho_\scf = \frac{1}{1+|\sigma|\brac{x}}, \quad \rho_\zf = \frac{|\sigma|\brac{x}}{1+|\sigma|\brac{x}}.\]
    The principal symbol in the semiclassical sc-b calculus satisfies
    \[\sigma_\hscb(\rho_\tf^{-2}P_\beta(h,\sigma))(x,\xi) = e^{-2i\beta}(\xi^2 + \rho_\scf^2) - e^{2i\arg(\sigma)}\rho_\zf^2 + \lilO_{x\to\infty}(1).\]
    Choosing $R_2$ large enough and using $\rho_\scf+\rho_\zf=1$, we see that $P_\beta(h,\sigma)$ is semiclassically sc-b elliptic in $x\geq R_2$, uniformly in $\arg(\sigma) \in [-\beta+\delta,\pi-\beta-\delta]$ (for any $\delta>0$).

    On the other hand, in $[-1,R_2]$ and for $|\sigma|$ small enough, the operator is semiclassically elliptic in the usual sense. Indeed, near $[R_1,R_2]$ the semiclassical principal symbol is given by
    \begin{equation*}
    \begin{split}
        \sigma_\h(P_\beta(h,\sigma))(x,\xi) &= \frac{e^{-2i\phi_\beta(x)}}{(1+ix\phi_\beta'(x))^2}\xi^2 + V(0,e^{i\phi_\beta(x)}x) - \sigma^2 \\
        &= e^{-2i\phi_\beta(x)}(\xi^2+x^{-2}) - \sigma^2 + \bigO(\epsilon\xi^2) + \lilO_{x\to\infty}(x^{-2}),
    \end{split}
    \end{equation*}
    where we used $|x\phi_\beta'(x)|=\bigO(\epsilon)$ and $V(0,e^{i\phi_\beta(x)}x) = e^{-2i\phi_\beta(x)}x^{-2} + \lilO_{x\to\infty}(x^{-2})$. Thus, choosing $R_1$ large enough and $\epsilon$ small enough, we can find $c_0>0$ such that 
    \[|\sigma_\h(P_\beta(h,\sigma))(x,\xi)|\geq C\brac{\xi}^2\]
    for all $x$ in a neighborhood of $[R_1,R_2]$ and $|\sigma|\leq c_0$. Finally, for $x\in[-1,R_1]$ we have $\phi_\beta(x)=0$ and
    \[\sigma_\h(P_\beta(h,\sigma))(x,\xi) = \xi^2 + V(0,x) - \sigma^2.\]
    Since $V(0,x)$ is strictly positive, we can again choose $c_0>0$ so that $\sigma_\h(P_\beta(h,\sigma)) \geq C\brac{\xi}^2$ in $x\in[-1,R_1]$, uniformly in $|\sigma|\leq c_0$.

    Combining the standard semiclassical elliptic estimates in $[-1,R_2]$ with semiclassical sc-b elliptic estimates in $[R_2,\infty)$, see Prop.\ref{prop_hscb_elliptic_estimate}, and recalling that the usual semiclassical Sobolev norms agree with the semiclassical sc-b norms on bounded subsets, we obtain the estimate
    \[\|\chi_+ u\|_{H^s_\scbnorm(\R)} \leq C\bigl(\|\tilde{\chi}_+P_\beta(h,\sigma)u\|_{\rho_\tf^2H^{s-2}_\scbnorm(\R)} + h^N\|\tilde{\chi}_+u\|_{H^{-N}_\scbnorm(\R)}\bigr),\]
    uniformly in $h\in (0,1]$ and $\sigma\in\{\sigma\in\C\setminus\{0\},\, |\sigma|\leq c_0,\, \arg(\sigma)\in[-\beta+\delta,\pi-\beta-\delta]\}$. The lemma follows by recalling that $h=(\ell(\ell+1))^{-\frac{1}{2}}\sim\brac{\ell}^{-1}$ and $P_\beta(h,\sigma) = h^2P_{\beta,\ell}(\frac{\sigma}{h})$.
\end{proof}

\begin{rmk}
    If we switch to the framework of \cite{stucker_QNM_Kerr}, i.e. work in coordinates that extend beyond the horizon, the behavior near asymptotically flat infinity remains largely unchanged. That is, the Schwarzschild spectral family at angular momentum $\ell$ still defines a semiclassical scattering-b transition operator, which is elliptic near infinity. This observation can be combined with propagation and radial estimates to prove the exclusion of resonances with $|\sigma|\leq c_0\ell$ in strips $\Im(\sigma) > -C$. Due to the use of radial point estimates, one is restricted to working in strips below the real axis. Indeed, for a given $\Im(\sigma)$, radial estimates only hold on Sobolev spaces above some threshold regularity. This makes it impossible to obtain estimates on a fixed Sobolev space as $|\sigma|\to\infty$ in a full sector below the real axis. To deal with the horizon, we thus apply complex scaling and use the method of \cite{zworski_large_parameter_scaling}. It may be possible to avoid complex scaling at the horizon by adapting the method of \cite{jezequel_res_counting} to our situation.
\end{rmk}

\subsection{Estimates near the black hole horizon}
\label{section_horizon_estimates}
In $\Re(z)\leq 0$, we will use a different approach to obtain elliptic estimates for the complex scaled operator. Unlike the estimates in Lemma \ref{lemma_infty_estimates}, these do not hold in a full disc $|\sigma| \leq c_0\sqrt{\ell(\ell+1)}$, but rather in an annulus of the form $K_0 \leq |\sigma| \leq c_0\sqrt{\ell(\ell+1)}$. Moreover, we must take the scaling angle $\beta$ sufficiently small and restrict $\sigma$ to a small sector $\arg(\sigma)\geq-\theta_0$ below the real axis.

Recall that the Regge-Wheeler potential decays exponentially as $\Re(z)\to-\infty$, i.e.\ as one approaches the black hole horizon. We follow the method of \cite{zworski_large_parameter_scaling}, which is designed for such potentials and was already used in \cite{saBarreto_zworski} to prove the absence of resonances in an annulus as above for the Schwarzschild-de Sitter black hole, see also \cite{bony_haefner_res_exp}. This requires using a different semiclassical rescaling of the operator and a complex scaling contour that depends on $|\sigma|,\ell$.

We consider the analytically continued Regge-Wheeler operator \eqref{regge-wheeler} in $\Re(z)\leq 0$. We take
\[K_0 \leq |\sigma| \leq c_0\sqrt{\ell(\ell+1)}, \quad h=|\sigma|^{-1}, \quad \lambda^2 = h^2\ell(\ell+1),\]
where the constants $K_0, c_0 > 0$ will be determined later. Notice that $h\in (0,K_0^{-1}]$ is indeed a small parameter, whereas $\lambda \in [c_0,\infty)$ is a large parameter. We define the semiclassical rescaling
\begin{equation}
\label{semiclassical_rescaling_horizon}
    P(h,\lambda,\omega) = h^2P_\ell\Bigl(\frac{\omega}{h}\Bigr) = (hD_z)^2 + h^2V_\ell(z) - \omega^2, \quad\text{where}\quad |\omega|=1.
\end{equation}
From Lemma \ref{lemma_potential}, it follows that
\[h^2V_\ell(z) \sim Ce^{\tfrac{z}{2m}}(\lambda^2 + \bigO(h^2)), \quad\text{as } \Re(z)\to -\infty.\]
Thus, in $\Re(z)<0$ we are dealing with an exponentially decaying potential, which depends symbolically on an additional large parameter $\lambda$. We can therefore use the results of \cite[\S 4]{zworski_large_parameter_scaling} to obtain semiclassical elliptic estimates for a complex scaled version of $P(h,\lambda,\omega)$. This relies on a complex scaling contour that depends on $\lambda$, and thus on $|\sigma|,\ell$. Indeed, away from complex scaling, i.e.\ for $z$ real and bounded away from infinity, $P(h,\lambda,\omega)$ is semiclassically elliptic for $\lambda$ large enough by the positivity of the potential. However, as $|z|\to\infty$ along the complex scaling contour, we have $z\sim e^{i\beta}x$, where $x\to-\infty$, so the exponential factor in the potential is wildly oscillating. Since the potential also grows as $\lambda^2$, we have no chance of obtaining ellipticity uniformly in $\lambda$ on a fixed contour. Instead, we must ensure that the potential is sufficiently small, uniformly in $\lambda$, once we enter the complex scaling region. This is achieved by only starting the complex deformation when $|x|$ is on the order of $2m\log(\lambda^2)$. More precisely, for $\lambda^2=|\sigma|^{-2}\ell(\ell+1)$ large enough, we set $\Gamma_{\beta,|\sigma|,\ell} = G_{\beta,|\sigma|,\ell}((-\infty,0])$, where $G_{\beta,|\sigma|,\ell}:(-\infty,0] \to \C$ satisfies
\begin{equation}
\label{contour_horizon}
    G_{\beta,|\sigma|,\ell}(x) = \begin{cases} x, \quad &x \geq -2m\log(\lambda^2) + C_1, \\ x + i\tan(\beta)\bigl(x+2m\log(\lambda^2)\bigr), \quad &x \leq -2m\log(\lambda^2) - C_2.\end{cases}
\end{equation}
Note that $\arg(G_{\beta,|\sigma|,\ell}(x)) \to \pi+\beta$ as $x\to-\infty$. The semiclassical estimates for the operator in \eqref{semiclassical_rescaling} restricted to the scaling contour, translate to the following uniform estimates for the complex scaled Regge-Wheeler operator $P_{\beta,\ell}(\sigma) = P_\ell(\sigma)|_{\Gamma_{\beta,|\sigma|,\ell}}$, where we once again identify $H^s(\Gamma_{\beta,|\sigma|,\ell})$ with $H^s(\R)$ via the coordinate chart $G_{\beta,|\sigma|,\ell}^{-1}$.
\begin{lemma}
\label{lemma_horizon_estimates}
    Let $\chi_-,\tilde{\chi}_-\in C^\infty(\R)$ satisfy 
    \[\supp(\chi_-),\,\supp(\tilde{\chi}_-) \subset (-\infty,0),\quad \chi_-=1 \,\text{ on }\, (-\infty,-\tfrac{1}{2}),\quad \tilde{\chi}_- = 1 \,\text{ on }\, \supp(\chi_-).\] 
    For $\beta\in(0,\frac{\pi}{2})$ sufficiently small, there exists $\theta_0\in(0,\beta)$ and $K_0,c_0>0$ such that for all $s,N\in\R$ the following estimate holds:
    \begin{equation}
        \|\chi_- u\|_{H^s_{|\sigma|^{-1}}(\R)} \leq C_{s,N}\bigl(|\sigma|^{-2}\|\tilde{\chi}_-P_{\beta,\ell}(\sigma)u\|_{H^{s-2}_{|\sigma|^{-1}}(\R)} + |\sigma|^{-N}\|\tilde{\chi}_-u\|_{H^{-N}_{|\sigma|^{-1}}(\R)}\bigr),
    \end{equation}
    uniformly in $\ell\in\N_0$ and $\sigma\in\{\sigma\in\C\setminus\{0\},\, K_0\leq |\sigma|\leq c_0\sqrt{\ell(\ell+1)},\, |\arg(\sigma)|\leq\theta_0\}$.
\end{lemma}
\begin{proof}
    The estimate will follow from the semiclassical ellipticity of the complex scaled operator $P(h,\lambda,\omega)|_{\Gamma_{\beta,\lambda}}$, where $P(h,\lambda,\omega)$ is the operator in \eqref{semiclassical_rescaling_horizon} and $\Gamma_{\beta,\lambda}$ is the contour in \eqref{contour_horizon}.
    In order to match the convention of \cite{zworski_large_parameter_scaling}, we first rewrite our operator using the coordinate $\tilde{z} = -\frac{z}{4m}$ and define
    \begin{equation}
    \label{semiclassical_rescaling_horizon_2}
        \tilde{P}(h,\lambda,\tilde{\omega}) = (4m)^2P\bigl(h,\lambda,\tfrac{\tilde{\omega}}{4m}\bigr) = (4mh)^2P_\ell\bigl(\tfrac{\tilde{\omega}}{4mh}\bigr) = D_{\tilde{z}}^2 + (4mh)^2V_\ell(-4m\tilde{z}) - \tilde{\omega}^2.
    \end{equation}
    Here, $\arg(\tilde{\omega}) \in (-\beta,\pi-\beta)$ and $|\tilde{\omega}|=4m$. By Lemma \ref{lemma_potential} the potential satisfies 
    \[(4mh)^2V_\ell(-4m\tilde{z}) = e^{-2\tilde{z}}g(\lambda,\tilde{z}), \quad \Re(\tilde{z})\geq 0,\]
    where $g(\lambda,\tilde{z})$ is holomorphic in a sector around $[0,\infty)$, real-valued for $\tilde{z}$ real, and satisfies
    \[\lim_{\Re(\tilde{z})\to\infty} = \tfrac{4}{e}\bigl(\lambda^2+h^2\bigr).\]
    Moreover, $g(\lambda,\tilde{z})$ is an elliptic symbol in $\lambda$. In fact,
    \[c\lambda^2 \leq |g(\lambda,\tilde{z})| \leq C\lambda^2, \quad\text{and}\quad |\p_{\tilde{z}}^kg(\lambda,\tilde{z})| \leq C_k\lambda^2, \quad \forall\,k\in\N.\]
    
    The operator $\tilde{P}(h,\lambda,\tilde{\omega})$ is thus precisely of the form studied in \cite[\S 4]{zworski_large_parameter_scaling}. In terms of the coordinate $\tilde{z}$, the complex scaling contour in \eqref{contour_horizon} coincides with the $\lambda$-dependent contour introduced by Zworski, that is, $z\in\Gamma_{\beta,|\sigma|,\ell}$ implies $\tilde{z} \in \tilde{\Gamma}_{\beta,\lambda}$ -- the contour defined for large $\lambda$ by \cite[(4.4)]{zworski_large_parameter_scaling}.
    
    Denote by $\tilde{p}_\beta(\lambda,\tilde{\omega},x,\xi)$ the semiclassical principal symbol of $\tilde{P}_\beta(h,\lambda,\tilde{\omega}) = \tilde{P}(h,\lambda,\tilde{\omega})|_{\tilde{\Gamma}_{\beta,\lambda}}$. \cite[Lemma 4.3]{zworski_large_parameter_scaling} shows that for $C_1,C_2$ chosen large enough and under an additional assumption on the contour in the transition region, see \cite[(4.5)]{zworski_large_parameter_scaling}, $\tilde{p}_\beta(\lambda,\tilde{\omega},x,\xi)$ is uniformly elliptic. More precisely, for fixed $\beta\in (0,\frac{\pi}{2})$ small enough, there exists $\theta_0,\lambda_0 > 0$, such that
    \[|\tilde{p}_\beta(\lambda,\tilde{\omega},x,\xi)| \geq c(1+|\xi|^2+e^{-2x}\lambda^2), \quad \forall\, (x,\xi)\in T^*(0,\infty)\]
    uniformly in $\lambda\in [\lambda_0,\infty)$, $|\arg(\tilde{\omega})| \leq \theta_0$ with $|\tilde{\omega}|=4m$. Since the principal symbol also satisfies
    \[|\p_x^j\p_\xi^k \tilde{p}_\beta(\lambda,\tilde{\omega},x,\xi)| \leq C_{\alpha,\gamma}(1+|\xi|^2+e^{-2x}\lambda^2)^{1-\frac{k}{2}}, \quad \forall\, (x,\xi)\in T^*(0,\infty),\, \lambda\in[\lambda_0,\infty),\]
    for any $j,k\in\N_0$, its reciprocal $q_\beta(\lambda,\tilde{\omega},x,\xi) = \tilde{p}_\beta(\lambda,\tilde{\omega},x,\xi)^{-1}$ satisfies symbolic estimates
    \[|\p_x^j\p_\xi^k q_\beta(\lambda,\tilde{\omega},x,\xi)| \leq C_{\alpha,\gamma}(1+|\xi|^2+e^{-2x}\lambda^2)^{-1-\frac{k}{2}} \leq C_{\alpha,\gamma}\brac{\xi}^{-2-k}, \quad \forall\, (x,\xi)\in T^*(0,\infty),\]
    uniformly in $\lambda\in [\lambda_0,\infty)$. Thus, we can construct a semiclassical parametrix for $\tilde{P}_\beta(h,\lambda,\tilde{\omega})$ whose symbol seminorms are uniform in $\lambda$. This leads to a semiclassical elliptic estimate
    \[\|\chi u\|_{H_h^s(\R)} \leq C\bigl(\|\tilde{\chi}\tilde{P}_\beta(h,\lambda,\tilde{\omega})u\|_{H_h^{s-2}(\R)} + h^N\|\tilde{\chi} u\|_{H_h^{-N}(\R)}\bigr),\]
    which is uniform in $\lambda\in[\lambda_0,\infty)$, $h\in(0,h_0]$ and $|\arg(\tilde{\omega})| \leq \theta_0$ with $|\tilde{\omega}|=4m$. Here, $\chi_-,\tilde{\chi}_- \in C_c^\infty(\R)$ are the cutoff functions from the statement of the lemma. Using the relation between $\tilde{P}(h,\lambda,\tilde{\omega})$ and the original Regge-Wheeler operator $P_\ell(\sigma)$, see \eqref{semiclassical_rescaling_horizon_2}, and recalling that $|\sigma| \in [K_0,c_0\sqrt{\ell(\ell+1)}]$ implies $h\in (0,K_0^{-1}]$ and $\lambda \in [c_0^{-1},\infty)$, we obtain the uniform estimate in the lemma for $c_0=\lambda_0^{-1}$ and $K_0=h_0^{-1}$.
\end{proof}

\subsection{Exclusion of low energy modes and resonance counting}
\label{section_weyl_law}
Combining the estimates of \S \ref{section_infty_estimates} in $x>0$ with the estimates of \S \ref{section_horizon_estimates} in $x<0$, leads to global elliptic estimates for the complex scaled operator. For an appropriate range of $\sigma,\ell$, the error term can be absorbed and we obtain invertibility of $P_{\beta,\ell}(\sigma)$.
Introducing a cutoff function, the estimates on the more precise spaces of Lemmas \ref{lemma_infty_estimates} and \ref{lemma_horizon_estimates}, in particular, imply a bound for the cutoff resolvent acting on $L^2(\R)$.

\begin{prop}
\label{prop_annulus_exclusion}
    For any fixed cutoff function $\chi\in C_c^\infty(\R)$, there exist $c_0,\ell_0,\theta_0 > 0$ such that the cutoff resolvent for the Regge-Wheeler operator satisfies
    \[\|\chi P_\ell(\sigma)^{-1}\chi u\|_{L^2(\R)} \leq C\|u\|_{L^2(\R)}, \quad \forall\, u\in L^2(\R)\]
    uniformly in $\ell\geq\ell_0$ and $\sigma\in\C\setminus\{0\}$ with $|\arg(\sigma)| \leq \theta_0$, $K_0\leq |\sigma| \leq c_0\sqrt{\ell(\ell+1)}$.
\end{prop}
\begin{proof}
    We use the complex scaling contour from \S\ref{section_infty_estimates} in $\Re(z)>0$ and the contour from \S\ref{section_horizon_estimates} in $\Re(z)<0$. Thus, we define
    \[\Gamma_\beta = H_{\beta,|\sigma|,\ell}(\R), \quad\text{where}\quad H_{\beta,|\sigma|,\ell}(x) = \begin{cases} F_\beta(x), \quad &x\in[0,\infty), \\ G_{\beta,|\sigma|,\ell}(x), \quad &x\in(-\infty,0], \end{cases}\]
    see \eqref{contour_infty} and \eqref{contour_horizon}.
    
    We denote by $P_{\beta,\ell}(\sigma) = P_\ell(\sigma)|_{\Gamma_\beta}$ the complex scaled operator, where we choose $\beta$ sufficiently small as in Lemma \ref{lemma_horizon_estimates}. Letting $\ell\in\N_0$ and $\sigma\in\C\setminus\{0\}$, $K_0\leq |\sigma|\leq c_0\sqrt{\ell(\ell+1)}$, $|\arg(\sigma)|\leq\theta_0$ for $K_0,c_0,\theta_0$ as in Lemma \ref{lemma_horizon_estimates}, the estimates there in particular imply that
    \[\|\chi_-u\|_{L^2(\R)} \leq C\bigl(\|\tilde{\chi}_-P_{\beta,\ell}(\sigma)u\|_{L^2(\R)}\bigr) + |\sigma|^{-N}\|\tilde{\chi}_-u\|_{L^2(\R)}\bigr).\]
    Similarly, the estimates in Lemma \ref{lemma_infty_estimates} imply that
    \[\|\chi_+u\|_{L^2(\R)} \leq C\bigl(\|\tilde{\chi}_+P_{\beta,\ell}(\sigma)u\|_{L^2(\R)} + \brac{\ell}^{-N}\|\tilde{\chi}_+u\|_{L^2(\R)}\bigr),\]
    where we used $\rho_\tf \geq |\sigma| \geq K_0$, see Remark \ref{rmk_exotic_norms}.
    
    Recalling that $\chi_+=1$ on $[-\frac{1}{2},\infty)$ and $\chi_-=1$ on $(-\infty,-\frac{1}{2}]$, we can thus estimate
    \[\|u\|_{L^2(\R)} \leq \|\chi_+u\|_{L^2(\R)} + \|\chi_-u\|_{L^2(\R)} \leq C\bigl(\|P_{\beta,\ell}(\sigma)u\|_{L^2(\R)} + \min(|\sigma|,\brac{\ell})^{-N}\|u\|_{L^2(\R)}\bigr).\]
    Choosing $\ell_0, K_0$ large enough and taking $\ell\geq\ell_0$, $|\sigma|\geq K_0$, we can absorb the error term into the left-hand side and obtain
    \[\|u\|_{L^2(\R)} \leq C\|P_{\beta,\ell}(\sigma)u\|_{L^2(\R)}.\]
    This shows that the kernel of the complex scaled operator is empty, and thus, as an index zero Fredholm operator, $P_{\beta,\ell}(\sigma)$ is invertible for $\ell,\sigma$ as in the proposition, with uniform resolvent estimates of the form
    \[\|P_{\beta,\ell}(\sigma)^{-1}f\|_{L^2(\R)} \leq C\|f\|_{L^2(\R)}, \quad \forall f\in L^2(\R).\]

     Note that $H_{\beta,|\sigma|,\ell}(x) = x$ for $-2m\log(\lambda^2) + C_1 \leq x \leq R_1$, where $R_1$ can be taken arbitrarily large and $\lambda = |\sigma|^{-1}\sqrt{\ell(\ell+1)} \geq c_0^{-1}$. Thus, choosing $c_0$ smaller if necessary, we can ensure that the support of the cutoff function $\chi$ lies in the region where $\Gamma_\beta$ coincides with $\R$. The meromorphic continuation of the cutoff resolvent is then given by
     \[\chi P_\ell(\sigma)^{-1}\chi = \chi P_{\beta,\ell}(\sigma)^{-1}\chi.\]
     Thus, the cutoff resolvent estimate follows from the above estimate for $P_{\beta,\ell}(\sigma)^{-1}$.
\end{proof}

By Proposition \ref{prop_annulus_exclusion}, there are no Schwarzschild quasinormal modes at large angular momentum $\ell$ within a sector $\arg(\sigma)\geq-\theta_0$ satisfying $K_0\leq |\sigma| \leq c_0\sqrt{\ell(\ell+1)}$. We can now extend this to $|\sigma| \leq c_0\sqrt{\ell(\ell+1)}$ by using the low energy results of \cite[\S 5]{stucker_QNM_Kerr}. Indeed, by \cite[Thm.\ 1.6]{stucker_QNM_Kerr}, there are only finitely many QNM -- for the full operator, i.e.\ without restricting to fixed angular momentum -- in the disc $|\sigma|\leq K_0$. Since all QNM have finite multiplicity, there must exist some $\ell_0$ large enough, such that $\{\arg(\sigma)\geq-\theta_0,\,|\sigma|\leq K_0\}$ contains no QNM of angular momentum $\ell\geq\ell_0$. Thus, increasing $\ell_0$ if necessary, Proposition \ref{prop_annulus_exclusion} extends to the full range $|\sigma|\leq c_0\sqrt{\ell(\ell+1)}$.

\begin{proof}[Proof of Thm.\ \ref{thm_exclusion}]
    We use coordinates $(t_*,r,\omega)$ for which the Schwarzschild metric extends beyond the horizon to a manifold $\R_{t_*}\times X$ and denote by $P^\hor(\sigma)$ the spectral family with respect to the time-coordinate $t_*$, see \S\ref{section_qnm_comparison} below. The proof of \cite[Thm.\ 1.6]{stucker_QNM_Kerr} shows that for a cutoff function $\chi\in C_c^\infty((2m,\infty)_r)$, the cutoff resolvent satisfied the estimate
    \[\|\chi P^\hor(\sigma)^{-1}\chi f\|_{L^2(X)} \leq C\|\chi f\|_{L^2(X)},\]
    uniformly in a small disc $\sigma\in\C\setminus i\R_{\leq 0}$, $|\sigma|\leq\epsilon$.
    Note that the resolvent estimates in \cite{stucker_QNM_Kerr} take place on scattering-b transition Sobolev spaces of high enough regularity $s$. However, for functions supported away from infinity these spaces are equivalent to the usual Sobolev spaces of regularity $s$. Moreover, the operator $P^\hor(\sigma)$ is elliptic in $\{r>2m\}$. Thus, away from the horizon we can use an elliptic parametrix to extend the high regularity estimates to $L^2(X)$.

    Since the finite rank poles of $\chi P^\hor(\sigma)^{-1}\chi$ form a discreet subset of $\epsilon\leq |\sigma|\leq K_0$, there must be some $\ell_0$ large enough such that the cutoff resolvent restricted to angular momentum $\ell\geq\ell_0$ has no poles in this region. And by the analyticity of $\sigma\to \chi P^\hor(\sigma)^{-1}\chi$, we obtain the estimate
    \[\|\chi P^\hor(\sigma)^{-1}\chi(fY_{\ell,m})\|_{L^2(X)} \leq \|\chi fY_{\ell,m}\|_{L^2(X)}, \quad \forall\, f\in L^2((2m,\infty)_r),\, \ell\geq\ell_0,\]
    uniformly in $|\sigma| \in [\epsilon,K_0]$. These uniform estimates in $|\sigma|\leq K_0$ lead to corresponding estimates for the cutoff Regge-Wheeler resolvent, see Lemma \ref{lemma_resolvent_relation} below. Combining this with Proposition \ref{prop_annulus_exclusion}, gives the desired uniform estimates in $\arg(\sigma) \geq -\theta_0$, $|\sigma|\leq c_0\sqrt{\ell(\ell+1)}$ and $\ell\geq\ell_0$. 
\end{proof}

Having excluded angular momentum $\ell$ quasinormal modes from $|\sigma|\leq c_0\ell$ for $\ell$ large enough, Thm.\ \ref{thm_counting} is now an immediate consequence of the work of Hitrik and Zworski \cite{hitrik_zworski} describing the high angular momentum QNM generated by the photon sphere.
\begin{proof}[Proof of Thm.\ \ref{thm_counting}]
    We take $\theta>0$ small and denote the quasinormal modes at angular momentum $\ell$ by $\QNM_\ell$. For any $\epsilon>0$ and all $\ell\geq\ell_0$ large, \cite[Thm.\ 2]{hitrik_zworski} shows that
    \begin{equation}
    \label{hitrik_zworski}
        \QNM_\ell\cap\{|\sigma|\geq\epsilon\ell\} = \{\lambda_{\ell,n},\,n\in\N_0\}, \quad\text{where}\quad \lambda_{\ell,n} = (\ell+\tfrac{1}{2})G\Bigl(2\pi\frac{n+\tfrac{1}{2}}{\ell+\tfrac{1}{2}},\frac{1}{\ell+\tfrac{1}{2}}\Bigr),
    \end{equation}
    for an analytic symbol $G(x,h)\sim\sum_jh^jG_j(x)$. Combining this with our resonance exclusion result, Thm.\ \ref{thm_exclusion}, shows that, in fact, all resonances at angular momentum $\ell\geq\ell_0$ are given by \eqref{hitrik_zworski}. 
    
    For any fixed $\ell$, the poles of $P_{\beta,\ell}(\sigma)^{-1}$ in $\arg(\sigma)>-\theta$ form a bounded subset, i.e. satisfy $|\sigma|\leq C$ for some $C>0$. Thus, by the discreteness of the set of resonances, the sector $\arg(\sigma)>-\theta$ contains only finitely many resonances of angular momentum $\ell\leq\ell_0$. The analytic symbol of Hitrik-Zworski thus describes all resonances modulo a finite number. The asymptotics for the resonance counting function now follow exactly as in the Schwarzschild-de Sitter case treated in \cite{hitrik_zworski}, see in particular \cite[(1.5)]{hitrik_zworski} for the origin of the constant $C_{\theta,m}$.
\end{proof}

\subsection{Effect of cutting off near the horizon}
\label{section_qnm_comparison}
In the previous subsections, we have treated Schwarzschild quasinormal modes as resonances of the Regge-Wheeler potential, following \cite{bachelot,saBarreto_zworski,hitrik_zworski}. The Regge-Wheeler operator $P_\ell(\sigma)$ is related to the spectral family $e^{i\sigma t}\Box_g e^{-i\sigma t}$ for the Schwarzschild wave operator with respect to the time-coordinate $t$, see \eqref{spectral_family_t} and \eqref{spectral_family_t_modification}. This is a differential operator on the region exterior to the black hole horizon -- recall that under the change of coordinates in \eqref{x(r)}, $\R_x$ corresponds to $(2m,\infty)_r$.
We applied complex scaling to this operator both at asymptotically flat infinity ($x\to\infty$) and at the black hole horizon ($x\to-\infty$). The resonances at angular momentum $\ell$ are obtained as poles of the cutoff resolvent $\chi P_\ell(\sigma)^{-1}\chi$, where $\chi \in C_c^\infty(\R_x) = C_c^\infty((2m,\infty)_r)$ cuts off away from both infinity and the horizon.

In \cite{stucker_QNM_Kerr}, quasinormal modes for the Schwarzschild, and indeed Kerr, black hole are obtained by a slightly different approach. Using coordinates $(t_*,r,\omega)$, for which the metric extends across the horizon, the spectral family $P^\hor(\sigma) = e^{i\sigma t_*}\Box_g e^{-i\sigma t_*}$ with respect to the time coordinate $t_*$ is considered. This is a differential operator on $X = (r_0,\infty)_r\times\Sph^2_\omega$, where we choose an arbitrary location $r_0<2m$ inside the horizon for the boundary. We now only need to apply complex scaling at infinity, whereas at the horizon the radial point estimates of \cite{vasy_KdS} are used. This leads to a meromorphic continuation of the cutoff resolvent $\chi^\hor P^\hor(\sigma)^{-1}\chi^\hor$ (acting on spaces of extendable distributions), where $\chi^\hor \in C_c^\infty(\overline{X})$ cuts off away from infinity, but can be identically one across the horizon.

It is not a priori clear that the poles of these two cutoff resolvents coincide. In particular, cutting off away from the horizon could lead to cancellation of poles. Indeed, in \cite{hintz_xie_dual_resonances} it is shown that such pole cancellation does occur for the wave equation on exact de Sitter space when one cuts off away from the cosmological horizon. In this section, we will show that the Schwarzschild wave equation does not exhibit this phenomenon, and both definitions of quasinormal modes in fact coincide.

We begin by relating the spectral families for the different time-coordinates. As in \cite[\S 3.1]{stucker_QNM_Kerr}, we define $t_* = t + \psi(r)$, where $\psi'(r) = \mu(r)^{-1} - h(r)$ with $\mu(r)=1-\frac{2m}{r}$ and $h \in C^\infty(\R_+)$ chosen so that $h(r)=\mu(r)^{-1}$ for $r$ large and $dt_*$ is everywhere time-like. In the coordinates $(t_*,r,\omega)$, the Schwarzschild metric extends to $\{r>0\}$ and the wave operator then takes the form
\begin{equation}
\label{wave_op_t_*}
    \Box_g = (2h-\mu h^2)\p_{t_*}^2 - 2(1-\mu h)\p_{t_*}\p_{r} - r^{-2}\p_r r^2\mu\p_r - r^{-2}\Delta_\omega - (2r^{-1} - 2\tfrac{r-m}{r^2}h - \mu h')\p_{t_*}.
\end{equation}
We can further choose $h(r) = 1$ in a neighborhood of the horizon, see \cite[Lemma 3.1]{stucker_QNM_Kerr}, so that the wave operator has analytic coefficients near the horizon. We write
\[P^\hor(\sigma) = e^{i\sigma t_*}\Box_g e^{-i\sigma t_*}\]
for the corresponding spectral family.

Applying complex scaling near infinity, we obtain an operator $P^\hor_\beta(\sigma) = P^\hor(\sigma)|_{X_\beta}$, where $X_\beta$ coincides with $X$ in $\{r \in (r_0,R_1)\}$ for some large $R_1$, see \cite[\S 3.2]{stucker_QNM_Kerr}. By \cite[Prop.\ 3.11]{stucker_QNM_Kerr}, for any $s\in\R$ the complex scaled operator is invertible as a meromorphic family  in $\{\sigma\in\C\setminus\{0\},\, \arg(\sigma)\in (-\beta,\pi-\beta),\, \Im(\sigma) > \frac{1}{4m}(\frac{1}{2}-s)\}$ of bounded operators $P^\hor_\beta(\sigma)^{-1}:\Bar{H}^{s-1}(X_\beta)\to \Bar{H}^{s}(X_\beta)$. Note that the radial point estimates over the horizon require the regularity of the Sobolev spaces to be sufficiently high compared with $-\Im(\sigma)$. Cutting off away from the complex scaling region, this provides a meromorphic continuation of the cutoff resolvent. Note that $P^\hor_\beta(\sigma)$ is elliptic in $\{r>2m\}$. Thus, if we further cutoff away from the horizon with some $\chi\in C_c^\infty((2m,\infty)\times\Sph^2)$, the cutoff resolvent extends to a bounded operator $\chi P^\hor(\sigma)^{-1}\chi : H^s(X)\to H^{s+2}(X)$ for any $s\in\R$. This can be seen by using a parametrix for $P^\hor(\sigma)$ in $r>2m$. We can now relate this operator to the cutoff resolvent for the Regge-Wheeler operator.

\begin{lemma}
\label{lemma_resolvent_relation}
    For any $\chi\in C_c^\infty((2m,\infty)_r)$ and any spherical harmonic $Y_{\ell,m}$, we have
    \begin{equation}
    \label{resolvent_relation}
        \bigl(\chi P_\ell(\sigma)^{-1}\chi f\bigr)Y_{\ell,m} = re^{-i\sigma\psi(r)}\chi P^\hor(\sigma)^{-1}\chi e^{i\sigma\psi(r)}r^{-1}\mu^{-1}(fY_{\ell,m}),
    \end{equation}
    for all $f\in L^2((2m,\infty))$ and $\sigma\in\C\setminus\{0\}$. Moreover, the operator norms satisfy
    \[\|\chi P_\ell(\sigma)^{-1}\chi\|_{L^2(\R_x,dx)\to L^2(\R_x,dx)} \leq C\|\chi P^\hor(\sigma)^{-1}\chi|_{Y_{\ell,m}}\|_{L^2((2m,\infty)_r,r^2dr)\to L^2((2m,\infty)_r,r^2dr)}.\]
\end{lemma}
\begin{proof}
    From \eqref{spectral_family_t},\eqref{spectral_family_t_modification} the Regge-Wheeler operator satisfies for any $u\in C^\infty((2m,\infty)_r)$:
    \begin{equation}
    \label{operator_relation}
        \bigl(P_\ell(\sigma)u\bigr) Y_{\ell,m} = r\mu e^{i\sigma t}\Box_ge^{-i\sigma t}r^{-1}uY_{\ell,m} = r\mu e^{-i\sigma\psi(r)}P^\hor(\sigma)e^{i\sigma\psi(r)}r^{-1}uY_{\ell,m},
    \end{equation}
    where we used the definition of $P^\hor(\sigma)$ and $t_*=t+\psi(r)$. Recall that $\psi(r)$ is smooth and bounded for $r$ bounded away from $2m$. Now, for $\Im(\sigma)\gg 0$ both $P_\ell(\sigma)$ and $P^\hor(\sigma)$ are invertible on appropriate Sobolev spaces.
    In particular, for any $f\in C_c^\infty((2m,\infty))$, we can set set $u = P_\ell(\sigma)^{-1}f \in C^\infty((2m,\infty))$ in \eqref{operator_relation} and find
    \[e^{i\sigma\psi}r^{-1}\mu^{-1}fY_{\ell,m} = P^\hor(\sigma)e^{i\sigma\psi}r^{-1} P_\ell(\sigma)^{-1}fY_{\ell,m}.\]
    Since the left-hand side is compactly supported and smooth, we can apply $P^\hor(\sigma)^{-1}$ to find
    \[\bigl(P_\ell(\sigma)^{-1}f\bigr)Y_{\ell,m} = re^{-i\sigma\psi}P^\hor(\sigma)^{-1}e^{i\sigma\psi}r^{-1}\mu^{-1}fY_{\ell,m}, \quad \forall f\in C_c^\infty((2m,\infty)),\, \sigma\gg 0.\]
    Introducing cutoff functions and using density of $C_c^\infty$ in $L^2$, we see that \eqref{resolvent_relation} holds in $\Im(\sigma)\gg 0$. Since both sides possess meromorphic continuations to $\sigma\in\C\setminus\{0\}$, the relation continues to hold away from the poles.

    Note from \eqref{x(r)} that $L^2(\R_x,dx) = L^2((2m,\infty)_r,\mu^{-1}dr)$. Since we are cutting off to a compact subset of $(2m,\infty)$, all factors appearing in \eqref{resolvent_relation} are bounded on $\supp(\chi)$ away from both zero and infinity. Moreover, the $L^2$-norms on $L^2((2m,\infty)_r,\mu^{-1}dr)$ and $L^2((2m,\infty)_r,r^2dr)$ are equivalent on functions of compact support in $(2m,\infty)$. Thus, the bound on the operator norm follows immediately from \eqref{resolvent_relation}.
\end{proof}

This shows that quasinormal modes defined through the Regge-Wheeler operator coincide with the poles of $\chi P^\hor(\sigma)^{-1}\chi$ for any $\chi$ supported in $\{2m<r<\infty\}$, and thus form a subset of the poles of the complex scaled operator $P^\hor_\beta(\sigma)^{-1}$. It remains to see that cutting off near the horizon does not cancel any poles. Note that cutting off near infinity does not affect the pole structure, see the proof of \cite[Thm.\ 1.1]{stucker_QNM_Kerr}.

For convenience, we restrict to spherical harmonics and consider the operator $P^\hor_{\ell,\beta}(\sigma) = P^\hor_\beta(\sigma)|_{Y_{\ell,m}}$ on $I_\beta = f_\beta((r_0,\infty))$ for a complex scaling map $f_\beta(r) = e^{i\phi_\beta(r)}r$ with $\phi_\beta(r)=0$ in $r<R_1$ as in \cite[\S 2.1]{stucker_QNM_Kerr}. Then $P^\hor_{\ell,\beta}(\sigma)^{-1}: \Bar{H}^s(I_\beta)\to \Bar{H}^{s+1}(I_\beta)$ exists for $s$ large enough and is just the restriction of $P^\hor_\beta(\sigma)^{-1}$ to $\mathrm{span}_{H^s(I_\beta)}(Y_{\ell,m})$.
We now consider the Laurent expansion of $P^\hor_{\ell,\beta}(\sigma)^{-1}$ near a pole.

\begin{lemma}
\label{lemma_laurent_expansion}
    Assume that $\sigma_0 \in \C\setminus\{0\}$ is a pole of $P^\hor_{\ell,\beta}(\sigma)^{-1}: \Bar{H}^s(I_\beta)\to \Bar{H}^{s+1}(I_\beta)$,  where $\arg(\sigma_0)\in (-\beta,\pi-\beta)$ and $s>-\frac{1}{2}-4m\Im(\sigma_0)$. Then there exist $J,M_J \in \N$ such that
    \begin{equation}
    \label{laurent_expansion}
        P^\hor_{\ell,\beta}(\sigma)^{-1} = \sum_{k=1}^J \frac{A_k}{(\sigma-\sigma_0)^j} + B_{\mathrm{hol}}(\sigma), \quad\text{for } \sigma \text{ near } \sigma_0,
    \end{equation}
    where $B_{\mathrm{hol}}(\sigma):\Bar{H}^s(I_\beta)\to \Bar{H}^{s+1}(I_\beta)$ is holomorphic in $\sigma$ and the $A_k:\Bar{H}^s(I_\beta)\to \Bar{H}^{s+1}(I_\beta)$ are finite rank operators. Furthermore,
    \begin{equation}
    \label{top_deg_pole}
        A_J(f) = \sum_{j=1}^{M_J}\pair{\varphi_j}{f}u_j, \quad \text{for some linearly independent } \{u_j\} \subset \Bar{H}^{s+1}(I_\beta),\, \{\varphi_j\} \subset \dot{H}^{-s}(I_\beta),
    \end{equation}
    satisfying
    \begin{equation}
    \label{dual_resonant_states}
        P^\hor_{\ell,\beta}(\sigma_0)u_j = 0, \quad P^\hor_{\ell,\beta}(\sigma_0)^*\varphi_j = 0, \quad \forall\,j.
    \end{equation}
\end{lemma}
\begin{proof}
    Near $\sigma_0$, the operator $P^\hor_{\ell,\beta}(\sigma)^{-1}$ is meromorphic on the Sovolev spaces in question with poles of finite rank. Thus, the first part of the lemma follows immediately from the Laurent expansion around $\sigma_0$. Since $A_J:\Bar{H}^s(I_\beta)\to \Bar{H}^{s+1}(I_\beta)$ has finite rank, we can certainly represent $A_J$ as in \eqref{top_deg_pole} for a linearly independent set of $u_j\in \Bar{H}^{s+1}(I_\beta)$ and $\varphi_j\in \Bar{H}^{s}(I_\beta)^* = \dot{H}^{-s}(I_\beta)$. (The dot denotes the space of supported distributions, see \cite[\S 2.3]{stucker_QNM_Kerr}.)

    It remains to show \eqref{dual_resonant_states}. By the residue theorem, we have for a contour $\gamma_0$ encircling $\sigma_0$ and no other poles:
    \[A_J = \frac{1}{2\pi i}\int_{\gamma_0}(\sigma-\sigma_0)^{J-1}P^\hor_{\ell,\beta}(\sigma)^{-1}\,d\sigma.\]
    Now, note that $P^\hor_{\ell,\beta}(\sigma) = P^\hor_{\ell,\beta}(\sigma_0) + (\sigma-\sigma_0)Q(\sigma)$ for some $\sigma\to Q(\sigma) \in \Diff^1(I_\beta)$ holomorphic. Thus, on $\Bar{H}^s(I_\beta)$ we have
    \[P^\hor_{\ell,\beta}(\sigma_0)A_J = \frac{1}{2\pi i}\int_{\gamma_0}\bigl((\sigma-\sigma_0)^{J-1} - (\sigma-\sigma_0)^JQ(\sigma)P^\hor_{\ell,\beta}(\sigma)^{-1}\bigr)\,d\sigma = 0.\]
    Similarly, we find $A_JP^\hor_{\ell,\beta}(\sigma_0) = 0$ on $\Bar{H}^{s+2}(I_\beta)$. This leads to
    \[\sum_{k=1}^{M_J}\pair{\varphi_j}{f}P^\hor_{\ell,\beta}(\sigma_0)u_j = 0, \quad \sum_{k=1}^{M_J}\pair{\varphi_j}{P^\hor_{\ell,\beta}(\sigma_0)g}u_j = 0, \quad \forall\, f\in \Bar{H}^s(I_\beta),\, g\in \Bar{H}^{s+2}(I_\beta),\]
    which by the linear independence of $\{\varphi_j\}$ and $\{u_j\}$ gives \eqref{dual_resonant_states}.
\end{proof}

The $u_j$ can be thought of as resonant states and the $\varphi_j$ as dual resonant states associated to the resonance $\sigma_0$. However, note that in our case these objects live on the complex scaled contour $I_\beta$.
We will now study the behavior of these states near the event horizon, and thus away from where complex scaling takes place. We begin by showing that the resonant states cannot be identically zero in the exterior region. This follows from the analyticity of the resonant states near the horizon.

\begin{lemma}
\label{lemma_keldysh}
    Let $0\neq u\in \Bar{H}^{s+1}(I_\beta)$ satisfy $P^\hor_{\ell,\beta}(\sigma_0)u = 0$ with $\sigma_0,s$ as in Lemma \ref{lemma_laurent_expansion}. Then $\supp(u)\cap \{r>2m\} \neq \emptyset$.
\end{lemma}
\begin{proof}
    Since $s+1$ is above the threshold for high regularity radial point estimates over the horizon, microlocal radial, propagation and elliptic estimates show that $P^\hor_{\ell,\beta}(\sigma_0)u=0$ implies $u\in C^\infty(I_\beta)$, see for instance \cite[Remark 3.10]{stucker_QNM_Kerr}.
    In a neighborhood of $r=2m$, we have
    \[P^\hor_{\ell,\beta}(\sigma_0) = r^{-2}D_rr^2\mu D_r - \frac{4m}{r}\sigma D_r + \frac{\ell(\ell+1)}{r^2} - \Bigl(1+\frac{2m}{r}\Bigr)\sigma^2 + i\frac{2m}{r^3}\sigma,\]
    see \eqref{wave_op_t_*}, where we set $h(r)=1$ near the horizon. Working with the coordinate $y=r-2m$, we find
    \[rP^\hor_{\ell,\beta}(\sigma_0) = yD_y^2 + a(y)D_y + b(y),\]
    with $a,b$ analytic in a neighborhood of $y=0$. For such a generalized Keldysh operator, \cite[Thm.\ 1]{galkowski_zworski_keldysh} shows that $P^\hor_{\ell,\beta}(\sigma_0)u =0$ for $u$ smooth implies that $u$ is analytic in a neighborhood of $r=2m$, see also \cite{petersen_vasy_analyticity}. Propagation of analytic regularity in $\{r<2m\}$, see for instance \cite[Thm.\ 2.9.1]{hitrik_sjostrand_minicourse}, then shows that $u$ is analytic in $r\in(r_0,2m+\epsilon)$ for some $\epsilon>0$. Thus, $\supp(u)\subset [r_0,2m]$ would imply that $u=0$.
\end{proof}

We now consider the dual resonant states $\varphi_j$ and show that their support is contained in $\{r\geq 2m\}$. This is essentially a consequence of the finite speed of propagation for the Schwarzschild wave equation. 
\begin{lemma}
\label{lemma_dual_resonance_support}
    Let $\{\varphi_j\}$ be as in Lemma \ref{lemma_laurent_expansion}. Then $\supp(\varphi_j) \subset \{r\geq 2m\}$ for each $j$.
\end{lemma}
\begin{proof}
    Consider a forcing term $w\in C_c^\infty(\R_{t_*}\times[r_0,\infty)_r\times\Sph^2_\omega)$ with support inside the horizon: $\supp(w)\subset \{t_*\in (0,T),\, r< 2m\}$ for some $T>0$. Let $v$ be the unique solution to the inhomogeneous wave equation on Schwarzschild:
    \[\Box_gv = w, \quad v=0 \,\text{ in }\, t_*<0.\]
    Then by finite speed of propagation, $v$ is identically zero outside the horizon: $v(t_*,r,\omega) = 0$ for all $r>2m$ and all $t_*,\omega$. Taking $\Im(\sigma)\gg 0$, we can Fourier transform the wave equation in time, leading to $P^\hor(\sigma)\hat{v}(\sigma,r,\omega) = \hat{w}(\sigma,r,\omega)$, where
    \[\hat{v}(\sigma,r,\omega) = \int_\R e^{i\sigma t_*}v(t_*,r,\omega)\,dt_*, \quad \hat{w}(\sigma,r\omega) = \int_\R e^{i\sigma t_*}w(t_*,r,\omega)\,dt_*.\]
    Note that the convergence of the above integrals for large $\Im(\sigma)$ follows from simple energy estimates.
    As $P^\hor(\sigma)$ is invertible for $\Im(\sigma)\gg 0$, we find
    \[\hat{v}(\sigma,r,\omega) = P^\hor(\sigma)^{-1}\hat{w}(\sigma,r,\omega) = 0, \quad\text{on } \{r>2m\}.\]
    Since this holds for any forcing term $w$ with support in $r<2m$, we have for all $\Im(\sigma)\gg 0$:
    \begin{equation}
    \label{finite_speed}
        u\in \Bar{H}^s(X),\,\, \supp(u)\subset\{r<2m\} \,\,\implies\,\, \chi^\hor P^\hor(\sigma)^{-1}\chi^\hor u = 0 \,\,\text{ in } \{r>2m\},
    \end{equation}
    where we introduced an arbitrary cutoff function $\chi^\hor \in C_c^\infty([r_0,\infty))$ with $\chi=1$ on $r\leq R$ (note that $\chi^\hor u = u$). By the meromorphic continuation of the cutoff resolvent, \eqref{finite_speed} continues to hold for all $\sigma\in\C\setminus\{0\}$ away from the poles.
    
    Applying this to $u=fY_{\ell,m}$ for $f\in \Bar{H}^s(I_\beta)$ with $\supp(f)\subset[r_0,2m)$ shows that the complex scaled resolvent at angular momentum $\ell$ satisfies $P^\hor_{\ell,\beta}(\sigma)^{-1}(f) = 0$ for $r\in(2m,R)$, where $R<R_1$ -- the start of the complex deformation. Near the complex scaling region, $P^\hor_{\ell,\beta}(\sigma)^{-1}(f)$ extends to a holomorphic function on a neighborhood of $f_\beta((R_0,\infty)) \subset I_\beta \subset \C$ for some $R_0<R_1$, see the proof of \cite[Prop.\ 3.12]{stucker_QNM_Kerr}. Thus, choosing $R_0<R<R_1$, we find that in fact $P^\hor_{\ell,\beta}(\sigma)^{-1}(f) = 0$ on $f_\beta((2m,\infty))\subset I_\beta$. That is,
    \[\supp(P^\hor_{\ell,\beta}(\sigma)^{-1}(f)) \subset \{r\leq 2m\}, \quad \forall\, f\in \Bar{H}^s(I_\beta) \text{ with } \supp(f)\subset \{r<2m\}.\]
    
    With $f$ as above, we can now integrate around the pole at $\sigma=\sigma_0$ to find
    \[\sum_{k=1}^{M_J}\pair{\varphi_j}{f}u_j = \frac{1}{2\pi i}\int_{\gamma_0}(\sigma-\sigma_0)^{J-1}P^\hor_{\ell,\beta}(\sigma)^{-1}f\,d\sigma = 0, \quad\text{on } \{r>2m\}.\]
    Since the $u_j$ are linearly independent and not identically zero in $\{r>2m\}$ by Lemma \ref{lemma_keldysh}, we must have $\pair{\varphi_j}{f} = 0$ for all $j$. As this holds for any $f$ with $\supp(f)\subset \{r<2m\}$, the lemma follows.
\end{proof}

For a pole cancellation to occur when cutting off the resolvent with $\chi\in C_c^\infty((2m,\infty)_r)$, we must have $\chi u_j = \chi\varphi_j = 0$ for the resonant states and dual resonant states of Lemma \ref{lemma_laurent_expansion}. We have already seen that $\chi u_j\neq 0$ when $\chi$ has large enough support and it remains to consider the dual resonant states. Assume that $\chi\varphi_j = 0$ for all $\chi$ supported in the exterior region. Then $\supp(\varphi_j)\subset\{r\leq 2m\}$ is contained inside the horizon, so by the previous lemma $\varphi_j$ is supported precisely on the horizon. It remains to exclude this case. The calculation is somewhat similar to \cite[\S III.B]{hintz_xie_dual_resonances}.

\begin{lemma}
\label{lemma_no_delta_sol}
    Let $\sigma_0,s$ be as in Lemma \ref{lemma_laurent_expansion}. If $\varphi\in \dot{H}^{-s}(I_\beta)$ satisfies $P^\hor_{\ell,\beta}(\sigma_0)^*\varphi = 0$ and $\supp(\varphi) = \{2m\}$, then $\varphi=0$.
\end{lemma}
\begin{proof}
    For convenience, we set $h(r)=0$ near the horizon in \eqref{wave_op_t_*}. That is to say, for $H\in C^\infty([r_0,\infty)_r)$ with $H'=h$, we replace $\varphi$ by $e^{i\sigma_0H}\varphi$. This still satisfies the hypotheses of the lemma but with the spectral family $P^\hor_{\ell,\beta}(\sigma_0)$ obtained from the wave operator in \eqref{wave_op_t_*} with $h=0$. The formal adjoint of this operator takes the following form near the horizon:
    \[P^\hor_{\ell,\beta}(\sigma_0)^* = -r^{-2}\p_r r^2\mu\p_r + r^{-2}\ell(\ell+1) + 2i\Bar{\sigma}_0(\p_r + r^{-1}).\]
    Rewriting this operator in terms of the coordinate $y=r-2m$ and gathering terms by homogeneity, we find
    \begin{equation*}
    \begin{split}
        r^2P^\hor_{\ell,\beta}(\sigma_0)^* = &-(y\p_y)^2-y\p_y + \ell(\ell+1) + 4im\Bar{\sigma}_0(2y\p_y + 1) \\
        &- 2m(\p_yy\p_y - 4im\Bar{\sigma}_0\p_y) + 2i\Bar{\sigma}_0(y^2\p_y + y),
    \end{split}
    \end{equation*}

    Assume now that $\varphi$ is a non-trivial dual resonant state supported on the horizon. Since the distribution $\varphi$ has support at the point $r=2m$, it is a finite sum of derivatives of delta functions:
    \begin{equation}
    \label{dual_resonance_ansatz}
        \varphi = \sum_{k=0}^N a_k\delta^{(k)}(y), \quad\text{for some } N\in\N_0 \text{ and } a_k\in\C \text{ with } a_N\neq 0.
    \end{equation}
    Setting $P^\hor_{\ell,\beta}(\sigma_0)^*\varphi = 0$ and using $\p_y\delta^{(k)} = \delta^{(k+1)}$, $y\delta^{(k)} = -k\delta^{(k-1)}$, we find
    \begin{equation*}
    \begin{split}
        \sum_{k=0}^N a_k\Bigl(&\bigl(\ell(\ell+1) - k(k+1) - 4i\Bar{\sigma}_0m(2k+1)\bigr)\delta^{(k)} \\ 
        &+ 2m\bigl(k+1+4i\Bar{\sigma}_0m\bigr)\delta^{(k+1)} + 2i\Bar{\sigma}_0k^2\delta^{(k-1)}\Bigr) = 0.
    \end{split}
    \end{equation*}
    Gathering the terms multiplying $\delta^{(k)}$ for each $0\leq k\leq N+1$ and setting them to zero, results in $N+2$ equations for the $a_k\in\C$, $\sigma_0\in\C$ and $\ell\in\N_0$. From the coefficient of $\delta^{(N+1)}$, we find
    \[(N+1+4im\Bar{\sigma}_0)a_N = 0 \quad\implies\quad \sigma_0 = -i\frac{N+1}{4m},\]
    since $a_N\neq 0$ by assumption. Thus, a dual resonant state supported on the horizon would be associated to a resonance on the negative imaginary axis. Using this value of $\sigma_0$ in the remaining $N+1$ equations, we find
    \begin{equation}
    \label{delta_system}
    \begin{split}
        2ma_{N-1} &= \bigl((N+1)^2 + \ell(\ell+1)\bigr)a_N, \\
        2m(N-k+1)a_{k-1} &= \bigl((k+1)(N+1) + k(N-k) + \ell(\ell+1)\bigr)a_{k} - \frac{N+1}{2m}(k+1)^2a_{k+1}, \\
        \bigl(N+\ell(\ell+1)\bigr)a_0 &= \frac{N+1}{2m}a_1,
    \end{split}
    \end{equation}
    where the second equation holds for all $k\in\{1,\dots,N-1\}$. It remains to show that there is no non-trivial solution $\ell\in\N_0$, $a_k\in\C$ to this system.

    Consider first the case $N=0$, i.e.\ $\varphi = a_0\delta(y)$ with $a_0\neq 0$. Then $P^\hor_{\ell,\beta}(\sigma_0)^*\varphi = 0$ implies that $\sigma_0=-i\frac{1}{4m}$ and $\ell(\ell+1)-4im\Bar{\sigma_0} = \ell(\ell+1) + 1 = 0$, which is impossible for $\ell\in\N_0$.
    Now, assume that $N\geq 1$. We claim that
    \[\forall\, k\in\{1,\dots,N\}:\quad a_k\neq 0, \quad\text{and}\quad b_k := \frac{2ma_{k-1}}{a_k} > \frac{k(N+1)}{N-k+1}.\]
    For $k=N$, we have $a_N\neq 0$ by assumption and the first equation in \eqref{delta_system} gives
    \[b_N = (N+1)^2 + \ell(\ell+1) \geq (N+1)^2 > N(N+1)\]
    as claimed. Now assume for some $k\in\{1,\dots,N-1\}$ that the claim is true for $k+1$. Then $b_{k+1} > 0$ implies $a_k\neq 0$ and dividing the second equation in \eqref{delta_system} by $a_k$ gives
    \[(N-k+1)b_k = \bigl((k+1)(N+1) + k(N-k) + \ell(\ell+1)\bigr) - (N+1)(k+1)^2b_{k+1}^{-1}\]
    Using $\ell\geq 0$ and $b_{k+1}^{-1} < \frac{N-k}{(k+1)(N+1)}$, by assumption on $b_{k+1}$, we find
    \begin{equation*}
    \begin{split}
        (N-k+1)b_k &> (k+1)(N+1) + k(N-k) - (k+1)(N-k) \\
        &= k(N+1) + k + 1 > k(N+1),
    \end{split}
    \end{equation*}
    which proves the claim for $k$. Thus, we arrive at 
    \[b_1 = \frac{2ma_0}{a_1} > \frac{N+1}{N}.\]
    On the other hand, the third equation in \eqref{delta_system} implies that
    \[\frac{2ma_0}{a_1} = \frac{N+1}{N+\ell(\ell+1)} \leq \frac{N+1}{N}, \quad \forall\,\ell\in\N_0,\]
    which provides the desired contradiction.
\end{proof}
\begin{rmk}
    At least in spherical symmetry, the proof of Lemma \ref{lemma_no_delta_sol} suggests that, generically, black hole -- respectively cosmological -- spacetimes have no dual resonant states supported on the event -- or cosmological -- horizon. Indeed, without loss of generality we can set $a_N=1$ in the ansatz \eqref{dual_resonance_ansatz} for the dual resonant state $\varphi$. Acting on $\varphi$ with the adjoint of the Fourier transformed wave operator gives $N+2$ equations in the $N+2$ unknowns $\{a_k\}_{k=0}^{N-1}$, $\sigma_0$ and $\ell$. Since $\ell$ is restricted to lie in $\N_0$, one would expect a solution to exist only in special cases.
\end{rmk}

\begin{proof}[Proof of Thm.\ \ref{thm_pole_cancellation}]
    By Lemma \ref{lemma_resolvent_relation}, for any $\chi\in C_c^\infty((2m,\infty)_r)$ the poles of $\chi P_\ell(\sigma)^{-1}\chi$ coincide with the poles of $\chi P^\hor_\ell(\sigma)^{-1}\chi = \chi P^\hor_{\ell,\beta}(\sigma)^{-1}\chi$ for an appropriately chosen complex scaling contour $I_\beta$. Now, assume that $\sigma_0$ is a pole of the complex scaled resolvent $P^\hor_{\ell,\beta}(\sigma)^{-1}$. Then Lemma \ref{lemma_laurent_expansion} gives
    \[\frac{1}{2\pi i}\int_{\gamma_0}(\sigma-\sigma_0)^{J-1}\chi P^\hor_\ell(\sigma)^{-1}\chi\,d\sigma = \sum_{j=1}^{M_J}\chi u_j\pair{\chi\varphi_j}{\,\cdot\,},\]
    for some linearly independent resonant states $u_j$ and dual resonant states $\varphi_j$. Choosing $\chi = 1$ on a large enough subset of $(2m,\infty)$, we have $\chi u_j \neq 0$ by Lemma \ref{lemma_keldysh}. Similarly, combining Lemma \ref{lemma_dual_resonance_support} and Lemma \ref{lemma_no_delta_sol} gives $\chi\varphi_j \neq 0$ for all $j$. Thus, the residue of $(\sigma-\sigma_0)^{J-1}\chi P^\hor_\ell(\sigma)^{-1}\chi$ is non-zero, i.e.\ the cutoff resolvent has a pole at $\sigma_0$.
\end{proof}

\printbibliography

\end{document}